\documentclass[12pt]{article}
\usepackage[utf8x]{inputenc}
\usepackage{amsthm}
\usepackage{amsmath}
\usepackage{amssymb}
\usepackage{tikz}
\usepackage{enumitem}
\usepackage{mdwlist}
\usepackage{verbatim}
\usepackage[british]{babel}
\usepackage{scalefnt}
\usepackage{tikz,fullpage}
\usepackage{tkz-berge}
\usepackage{varioref}
\usepackage{cite}
\usepackage{caption}
\usepackage[subrefformat=parens,labelformat=parens]{subcaption}
\usepackage[T1]{fontenc}
\usepackage[twoside,a4paper]{geometry}
\usepackage{graphicx}

\usepackage[footnote,draft,silent,nomargin]{fixme}

\title{On the list chromatic index of graphs of tree-width 3 and maximum degree at least 7}
\author{Richard Lang\\ Universidad de Chile,\\ 
Santiago, Chile
 \\ \small rlang@dim.uchile.cl }

\theoremstyle{plain}
\numberwithin{equation}{section}
\numberwithin{figure}{section}
\newtheorem{conjecture}{Conjecture}
\newtheorem{theorem}{Theorem}
\newtheorem{lemma}{Lemma}[section]

\theoremstyle{definition}
\newtheorem{definition}{Definition}

\newcommand{\eb}[2]{v_{#1}w_{#2}}

\newcommand{\egen}[2]{{#1}{#2}}
\newcommand{\elgen}[2]{L(\egen{#1}{#2})}

\newcommand{\ch}{\mathrm{ch}}
\newcommand{\sm}{\setminus}

\begin{document}
\maketitle

\begin{abstract}
  Among other results, it is shown that 3-trees are $\Delta$-edge-choosable and that graphs of tree-width 3 and maximum
  degree at least 7 are $\Delta$-edge-choosable.
\end{abstract}


\section{Introduction}
\label{sec:introduction}
In this paper we analyse the list chromatic index, $\ch'(G),$ of simple graphs $G$ of tree-width 3 and high maximum degree.
Tree-width and path-width are in some sense measures of how much a graph resembles a tree and a path respectively (see Section
\ref{sec:configurations} for a proper definition). Our main
results are:
\begin{theorem}
\label{thm:simple-results}
 Let $G$ be graph with maximum degree $\Delta$. It holds that ${\ch'(G)=\Delta}$ if $G$ has
\begin{enumerate}
 \item \label{itm:simple-results-1} tree-width at most 3 and $\Delta \geq 7,$
 \item \label{itm:simple-results-2} path-width at most 3 and $\Delta \geq 6$ or
 \item \label{itm:simple-results-3} path-width at most 4 and $\Delta \geq 10$.
\end{enumerate}
\end{theorem}
A 3-tree is an edge-maximal graph of tree-width 3.
 \begin{theorem}
\label{thm:3-trees}
 Let $G$ be a 3-tree with chromatic index $\chi'(G);$ then $\ch'(G) = \chi'(G)$.
\end{theorem}

\subsection{The list colouring conjecture}

 List colourings are a generalisation of colourings introduced independently by 
 Vizing \cite{vizing76} and Erd\H{o}s, Rubin and Taylor \cite{ErRuTa79} in the seventies.
 An \emph{instance of list edge-colouring} consists of a graph $G$ and an \emph{assignment of lists} 
 $L:E(G) \rightarrow \mathcal{P}(\mathbb{N})$ that maps the edges of $G$
  to \emph{lists of colours} $L(e)$. 
 A function $\mathcal{C}:E(G) \rightarrow \mathbb{N}$ is called an
 \emph{$L$-edge-colouring} of $G,$ if $\mathcal{C}(e) \in L(e)$ for each $e \in E(G)$ and no two adjacent edges receive the same
 colour. $G$ is said to be \emph{$k$-edge-choosable}, if for each assignment of lists $L,$ where each list has
 a size of at least $k,$ there is an $L$-edge-colouring of $G$. The \emph{list chromatic index}, denoted by $\ch'(G),$ 
 is the smallest integer $k$
 for which $G$ is $k$-edge-choosable. The following conjecture is one of the central open
 problems in the field of list colouring.
 \begin{conjecture}[List edge-colouring conjecture]
  \label{con:list-colouring-conjecture}
  For all graphs $G$ it holds that $\ch'(G) = \chi'(G)$.
 \end{conjecture}
 While still open in general, Conjecture~\ref{con:list-colouring-conjecture} has been verified asymptotically by Kahn~\cite{journals/jct/Kahn96} and also
 for some particular families of graphs: Galvin  
 proved that $\ch'(G) = \Delta(G)$ for all bipartite graphs $G$ \cite{journals/jct/Galvin95}. Borodin et al. used this to show that 
 $\ch'(G) = \Delta(G)$, if the maximum average degree of $G$ is at most $\sqrt{2\Delta}$~\cite{journals/jct/BorodinKW97}.
 Ellingham and Goddyn \cite{journals/combinatorica/EllinghamG96} used a method of Alon and Tarsi to show that
 that every $d$-regular planar graph is $d$-edge-choosable.
 In 1999 Juvan, Mohar and Thomas showed that
 series-parallel graphs  are $\Delta$-edge-choosable~\cite{Juvan99listedge-colorings}. This family can also be characterised in terms of tree-width. 
 Series-parallel graphs are exactly the graphs of tree-width at most~2. Bruhn and Meeks suggested that 
 similar ideas could be applied to graphs of tree-width 3 and verified the list colouring conjecture for graphs of path-width at most 3 
 and maximum degree at least 6 [personal communication in 2012].

 \begin{figure}
\centering
\tikzstyle{vertex}=[circle,draw,inner sep=2pt,font=\tiny]
\tikzstyle{edge} = [draw,-]
\tikzstyle{weight} = [font=\small]
 \begin{tikzpicture}[scale=0.4]

    \foreach \pos/\name in {{(1,1)/v_3}, {(-2,-1)/v_2}, {(4,-1)/v_4},
                            {(1,4)/v_1},{(0,1)/v_5}}
       \node[vertex, align=center] (\name) at \pos {};
    \foreach \source/ \dest /\weight in {v_1/v_2/2, v_1/v_3/3, v_1/v_4/3, v_1/v_5/2, 
                                         v_2/v_3/3, v_2/v_4/3, v_2/v_5/3,
                                         v_3/v_4/2, v_3/v_5/2}
       \path[edge] (\source) -- node[weight] {} (\dest);

    \foreach \vertex  in {}
        \path node[selected vertex] at (\vertex) {};

 \end{tikzpicture}
\caption{A graph of tree-width 3 and chromatic index 5.}
\label{fig:counterexample}
\end{figure}
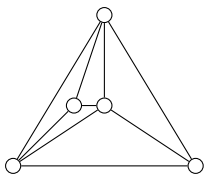
 The list colouring conjecture has been verified mostly for classes of graphs whose elements have chromatic index $\Delta(G)$.  
 But this does not hold for graphs of tree-width 3, see Figure~\ref{fig:counterexample}. 
 However, Nakano, Nishizeki and Zhou~\cite{DBLP:journals/jal/ZhouNN96} provided  the following result. 
 \begin{theorem}
  \label{thm:nishizeki}
  Let $G$ be a graph of tree-width $k$ and $\Delta(G) \geq 2  k$; then ${\chi'(G)=\Delta(G)}$.
 \end{theorem}
 Bearing this in mind we focused our research on graphs of tree-width 3 and high maximum degree. 
 
 \subsection{The approach}
 
 Let $G$ be a graph with a subset of edges $F \subset E(G)$ and an assignment of lists $L$ to the edges of $G$. 
 For an $L$-edge-colouring $\mathcal{C}$ 
 of $G - F$ we call a colour $c$
 of the list of an uncoloured edge $e \in F$ \emph{available}, if no edge adjacent to $e$ has already been coloured with $c$. The
 set of available colours of the edge $e$ is called \emph{list of remaining colours} and denoted by~$L^{\mathcal{C}}(e)$.
 
 Here is an outline of the proof of Theorem~\ref{thm:simple-results}. Let a graph $G$
 of tree-width 3 and high maximum degree $\Delta$ be given and lists of colours $L,$ each of size at least $\Delta,$
 be assigned to its edges. The tree-width will be used in combination with the maximum degree to locate a suitable substructure in $G$
 that consists of edges $F$. We then pursue a Vizing-like approach. More precisely ew use an inductional argument to find an $L$-edge-colouring $\mathcal{C}$ of the graph $G-F$. 
 Thus in order to extend $\mathcal{C}$ to an $L$-edge-colouring of $G$ we have to colour the edges $F$ from the lists of remaining
 colours $L^{\mathcal{C}}$. 
 We will prove the first two items of Theorem~\ref{thm:simple-results} 
 this way. In the proof of the third item we will have to find an $L$-edge-colouring $\mathcal{C}$ with certain properties. This is feasible
 by colouring an auxiliary graph $G^*$ of tree-width 3 and maximum degree $\Delta$. 
 At that point it will important that $G^*$ is $\Delta$-edge-colourable, which is asserted by Theorem~\ref{thm:nishizeki} if
 $\Delta$ is at least 6.
 
 The methods presented are used in~\cite{Lan13} to prove a list version of Vizing's theorem for graphs of tree-width 3 and to
 verify the list colouring conjecture for Halin graphs.
 The rest of the paper is organised as follows. In Section~\ref{sec:configurations} we will
 locate certain substructures for which we will solve the related instances of list edge colouring in Section~\ref{sec:instances}. 
 In Section~\ref{sec:results} we will combine these efforts to give proofs of the main results.

\section{Finding substructures}
 \label{sec:configurations}
 In this section we will identify some substructures that will arise within the graphs of our interest.
 \subsection{Bounded tree-width}
 For a graph $G$ a \emph{tree decomposition} $(T,\mathcal{V})$ consists of a tree $T$ and a collection 
 $\mathcal{V} = \{V_t \text{ : } t \in V(T) \}$ 
 of \emph{bags} $V_t \subset V(G)$ such that
 \begin{itemize}
  \item $V(G) = \bigcup_{t \in V(T)} V_t,$
  \item for each $vw \in E(G)$ there exists a $t \in V(T)$ such that $v,$ $w \in V_t$ and
  \item if $v \in V_{t_1} \cap V_{t_2},$ then $v \in V_t$ for each vertex $t$ that lies on the path connecting $t_1$ and $t_2$ in $T$.
 \end{itemize}
 A tree decomposition $(T,\mathcal{V})$ of $G$ has \emph{width} $k,$ if each bag has a size of at most $k+1$.
 The \emph{tree-width} of $G$ is the smallest integer $k$ for which there exists a width $k$ tree decomposition of $G$. As
 our later proofs are based on minimality it is important to mention that graphs of tree-width at most $k$ form a minor-closed family. 
  A \emph{path decomposition} is a tree decomposition $(T,\mathcal{V})$ in which the associated
 tree is a path, and the \emph{path-width} of $G$ is the minimum width over all path decompositions of $G$.

 The next definition presents the general substructure that we are looking for.

\begin{figure}
\centering
\tikzstyle{vertex}=[circle,draw,minimum size=14pt,inner sep=0pt]
\tikzstyle{edge} = [draw,-]
\tikzstyle{weight} = [font=\small,draw,fill           = white,
                                  text           = black]
 
 \begin{tikzpicture}[scale=1.1]
\tikzstyle{vertex}=[circle,draw,minimum size=4pt,inner sep=0pt,fill]
\tikzstyle{edge-dot} = [draw,dotted]
\tikzstyle{edge-lin} = [draw,-]
\tikzstyle{weight} = [font=\small,draw,fill           = white,
                                  text           = black]
    \foreach \pos/\name in {{(1,0)/v_1},  {(3,0)/v_3},
                            {(0,2)/w_1}, {(4,2)/w_3},{(5,2)/w_5},{(4,0)/u}}
       \node[vertex, align=center] (\name) at \pos {};
    \foreach \pos/\name in {{(0,-1)/x_0},{(2,-1)/x_1},{(5,-1)/x_4},{(3,-1)/x_3},{(1,-1)/x_2}}
       \node[circle,inner sep = 0pt] (\name) at \pos {};
    \node [circle,inner sep = 0, minimum size = 12 pt, below] (test) at (4,0) {$u$};

\node [circle] (dots) at (2,0) {...};
\node [circle] (dots2) at (2,2) {...};

\node[circle] (W) at (-1,2) {\large{$W$}};
\node[circle] (V) at (-1,0) {\large{$V$}};

    \foreach \source/ \dest in {v_1/w_1,
 w_1/u/4, 
                                        w_3/v_3,
                                         v_3/w_1, 
                                         w_5/v_1,w_5/v_1,w_5/v_3,v_3/u,v_1/x_1,v_1/x_0,v_3/x_3,v_3/x_1,v_1/x_2,v_1/w_3}
       \path[edge-dot] (\source) -- node[] {} (\dest);

    \foreach \source/ \dest in {u/w_1, u/w_3,w_5/u}
       \path[edge-lin] (\source) -- node[] {} (\dest);

  \draw (v_1) edge[out=-30,in=-150,dotted] (u);
  \draw (v_1) edge[out=-30,in=-150,dotted] (v_3);

      \path[draw,opacity=.1,line width=20,line cap=round, color = black] (v_1) --node[] {} (u);
      \path[draw,opacity=.1,line width=20,line cap=round, color = black] (w_1) --node[] {} (w_5);

\fill [color=black,opacity=0.1] (2.5,-1.2) ellipse (3 and 0.6);
 \end{tikzpicture} 
\caption{A useful substructure.}
\label{fig:substructure}
\end{figure}
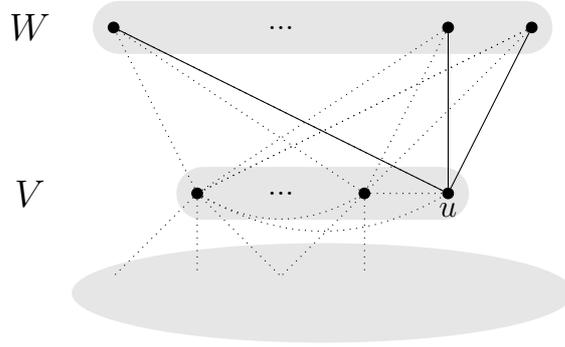
 \begin{definition}
  \label{def:substructure}
  For a graph $G$ and integers $k,l\in \mathbb N$ we call a triplet $(V,W,{u})$ that consists of two disjoint non-empty subsets $V,W \subset V(G)$ and
  a dedicated vertex ${u} \in V$ a \emph{$(k,l)$-substructure} if
 \begin{enumerate}[label=\emph{\alph*})]
  \item \label{itm:1}   $W$ is stable and $N(W) \subset V$,
  \item \label{itm:2}   each vertex of $W$ is connected to ${u},$
  \item \label{itm:3}   $\deg(w) \leq k$ for each $w \in W,$
  \item \label{itm:9}   $|V| \leq k+1,$
    \item \label{itm:8}   $N({u}) \subset (V \cup W),$
     \item \label{itm:4}   $\deg({u}) \geq l+2-k,$
       \item \label{itm:7}   $|W| \geq l+2 -2k$ and

  \item \label{itm:6}   there is a width $k$ tree decomposition $(T,\mathcal{V})$ of $G - W$ such 
                        that $V \subset V_t$ for a vertex $t \in V(T).$
 
  \end{enumerate} 
\end{definition}

 We start with a general result that can be extracted from \cite{journals/corr/abs-1110-4077}.

\begin{lemma}
 \label{lem:tree-width-k-configurations}
 For $l,$ $k \in \mathbb{N}$ with $l \geq 2k-1,$ let $G$ be a graph of tree-width at most $k$ 
 and $$\deg(v) + \deg(w) \geq \max(l,\deg(v),\deg(w))+2$$ 
 for each edge $vw\in E(G)$. Then $G$ has a $(k,l)$-substructure $(V,W,{u})$. 
 
 Furthermore, if $G$ has path-width of at most $k,$ then $|V| \leq k$, in the tree decomposition $(T,\mathcal{V})$
 associated with $(V,W,{u})$ the tree $T$ is a path and the vertex $t \in V(T)$ specified in Definition~\ref{def:substructure}(\ref{itm:6}) is a leaf.
\end{lemma}

 Here are some more definitions and an elemental lemma, that will be used only in the next proof.
 For a tree $T$, which is rooted in some vertex $r \in V (T ),$ we define the \emph{height} of
 any $t \in V(T)$ to be the distance from $r$ to $t$. If $(T, \mathcal{V})$ is a tree-decomposition of a graph $G$, then for any $v \in V (G),$ we define
 $t_v$ to be the unique vertex $t$ of minimum height such that $v \in V_t$. 
 For a connected graph $G$ we call $S \subset V(G)$ a cut-set, if the graph $G - S$ is not connected. 
 A proof of the next lemma can be found in \cite{Diestel00}.
\begin{lemma}
\label{lem:tree-cut}
 Let $G$ be a connected graph with a tree decomposition $(T,\mathcal{V})$. Then for any edge $t_1t_2 \in E(T)$ the intersection $V_{t_1} \cap V_{t_2}$ is a cut-set.
\end{lemma}
In other words, if $T_1$ and $T_2$ are the connected components of the forest $T -t_1t_2$, then the intersection $V_{t_1} \cap V_{t_2}$
separates the vertex sets $\bigcup_{t \in V(T_1)}(V_t)$ and $\bigcup_{t \in V(T_2)}(V_t)$ in $G$.

\begin{proof}[Proof of Lemma~\ref{lem:tree-width-k-configurations}]
 By the assumptions we have for all $vw \in E(G)$
  \begin{equation}
 \label{equ:tree-width-k-configurations}
  \deg(v) + \deg(w) \geq \max(l,\deg(v),\deg(w))+2 \geq l+2 \geq 2k+1.
 \end{equation}
 In particular, of any two adjacent vertices, at least one has degree at least $k+1$ (and $G$ has at least one vertex of degree at least $k+1$).  
 We define ${B} \subset V (G)$ to be the (non-empty) set of vertices of degree at least $k+1$.
 Then ${S} := V(G)\sm B$ is stable. 
 
 Fix a width $k$ tree decomposition $({T},\mathcal{V})$ of $G$ and root the associated tree ${T}$ in an 
 arbitrary vertex $r \in V({T})$. 
 Let ${{u}} \in B$ such that 
 $
 h(t_{{u}} ) = \max_{v \in B} h(t_v ).
$
   Define 
 ${T'}$ as the
 subtree of ${T}$ rooted at $t_{{u}},$ that is, the subgraph of ${T}$ induced by all vertices $t \in V({T})$ where the path from $t$ to the root $r$ contains $t_{{u}}$.  
 
 Set $X := \bigcup_{t \in V({T'})}V_t$ and $V:= N(X) \subset V_{t_u}$. Note that $|V| \leq k+1$.
 We have $B\cap X \subset V,$ since any $v \in (B \cap X)\sm  V$ would have $h(t_v)>h(t_{u})$, 
contrary to the choice of ${{u}}$.
  Consequently 
  \begin{equation}\label{X'}
  X\sm V \subset {S}.
 \end{equation} 
  
By definition of the tree decomposition, 
 no element of $X\sm V$ can appear in a bag indexed
 by a vertex $t \in V({T} - {T'})$. Since ${S}$ is stable this gives
  \begin{equation}\label{neighX'}
  N(X\sm V) \subset V.
 \end{equation}
   By definition of $t_{{u}},$ also ${{u}}$ does not appear in any bag $V_t$ of a vertex $t \in {T}- {T'}$. So, $N({u}) \subset X$.
 
    Set $W := N({{u}}) \sm V$. Then $W \subset X\sm U$. We claim that $(V,W,{u})$ is the desired substructure. To this end we check 
    if~\ref{itm:1}--\ref{itm:7} of Definition~\ref{def:substructure} hold.
    By~\eqref{X'},~\eqref{neighX'} and as $N(u) \subset X$, we can guarantee~\ref{itm:1}--\ref{itm:8}.
    
    Using the assumptions of the lemma and~\ref{itm:3}, we get
  $$\deg({{u}}) \geq l+2 - \deg(w) \geq l + 2 -k$$ and thus~\ref{itm:4}. Since $N({{u}}) \subset V \cup W$ we obtain 
  $$|W| \geq |N({{u}}) \sm (V \sm \{{{u}}\})| \geq l+2 -2k,$$
  which is as desired for ~\ref{itm:7}. 
     Note that the subtree $T'' = T - (T' - t_u)$ with $\mathcal{V}'' = \bigcup_{t \in T''} V_t$, is a tree-decomposition satisfying~\ref{itm:6}. 
  
  If $G$ has a path-width of at most $k$ we can assume that $T$ is a path and its root $r$ is a leaf. 
  Let $t'$ be the neighbour of $t_u$ in the subpath $T'.$ Without loss we can assume that $V_{t_u} \neq V_{t'}$. By Lemma~\ref{lem:tree-cut} the vertex set 
  $V_{t_{u}} \cap V_{t'}$ separates the vertices of $X'$ from the remaining vertices of $G$. Thus $N(W) \subset V_{t_{u}} \cap V_{t'}$ and consequently
   $|V| = |N(W)| \leq |V_{t_{u}} \cap V_{t'}| \leq k$.
   
  Finally, we want to show that $(T'',\mathcal{V}'')$ is path decomposition of $G-W$ where $T''$ starts at $r$ and ends in $t_u$.  
  If $X' \setminus W = \emptyset$, this is true. Otherwise,
  we need to transfer the vertices of $X' \setminus W$ into $(T'',\mathcal{V}'')$.   
  
  To this end 
  write $X' \sm W = \{x_1, \ldots, x_m\}$ and set $P$ to be the path $r\ldots t_u s_1 \ldots s_{m+1}$ with new vertices $s_1, \ldots, s_{m+1}$.
  Define bags $V''_{s_i} = \{x_i\} \cup V $ for $1 \leq i \leq m$ and $V''_{s_{m+1}} = V$. 
  Then the substructure $(V,W,u)$ with path decomposition $(P, \mathcal{V}'' \cup \{V''_{s_m} , \ldots , V''_{s_{m+1}} \})$ works as desired.
\end{proof}

 The next result provides the substructures for the proof of Lemma~\ref{lem:tree-width-3-delta-7}, which will be used
 to prove Theorem~\ref{thm:simple-results}(\ref{itm:simple-results-1}).

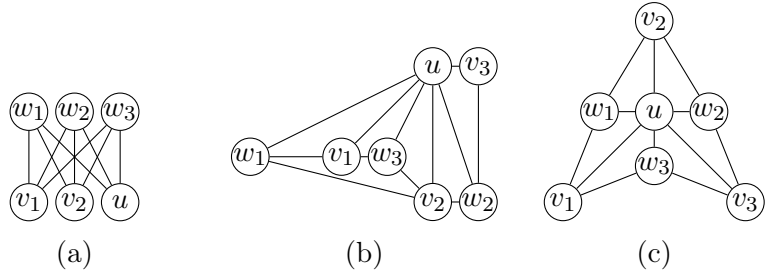
\begin{figure}
\centering
\tikzstyle{vertex}=[circle,draw,inner sep=0pt,font=\small, minimum size=14pt]
\tikzstyle{edge} = [draw,-]
\tikzstyle{weight} = [font=\small]

\begin{subfigure}[b]{0.2\textwidth}
\centering
 \begin{tikzpicture}[scale=0.6]
    \foreach \pos/\name in {{(1,0)/v_2}, {(0,0)/v_1},
                            {(2,0)/{u}},{(0,2)/w_1},{(1,2)/w_2},{(2,2)/w_3}}
       \node[vertex, align=center] (\name) at \pos {$\name$};
    \foreach \source/ \dest /\weight in {{u}/w_1/6, {u}/w_2/6, {u}/w_3/6,
                                          v_1/w_1/3, v_1/w_3/3, v_1/w_2/3,
                                          v_2/w_1/4, v_2/w_2/4, v_2/w_3/4} 
       \path[edge] (\source) -- node[weight] {} (\dest);
 \end{tikzpicture} \caption{}
\label{fig:tw3md7-conf-1}
\end{subfigure}
\begin{subfigure}[b]{0.3\textwidth}\centering
 \begin{tikzpicture}[scale=0.6]
    \foreach \pos/\name in {{(1,1)/v_2}, {(-1,2)/v_1}, {(2,4)/v_3},
                            {(1,4)/{u}},{(-3,2)/w_1},{(2,1)/w_2},{(0,2)/w_3}}
       \node[vertex, align=center] (\name) at \pos {$\name$};
    \foreach \source/ \dest /\weight in {{u}/v_1/3, {u}/v_2/4, {u}/v_3/2, {u}/w_1/6, {u}/w_2/6, {u}/w_3/6,
                                          v_1/w_1/3, v_1/w_3/3, 
                                          v_2/w_1/4, v_2/w_2/4, v_2/w_3/4,
                                         v_3/w_2/2} 
       \path[edge] (\source) -- node[weight] {} (\dest);
 \end{tikzpicture} \caption{}
\label{fig:tw3md7-conf-2}
\end{subfigure}
\begin{subfigure}[b]{0.2\textwidth}\centering
 \begin{tikzpicture}[scale=0.6]
    \foreach \pos/\name in {{(1,1)/{u}}, {(-1,-1)/v_1}, {(3,-1)/v_3},
                            {(1,3)/v_2},{(-0.2,1)/w_1},{(2.2,1)/w_2},{(1,-0.2)/w_3}}
       \node[vertex, align=center] (\name) at \pos {$\name$};
    \foreach \source/ \dest /\weight in {v_2/{u}/3, v_2/w_1/3, v_2/w_2/3, v_3/w_3/3,
                                          v_1/w_1/3, v_1/w_3/3, v_1/{u}/3, {u}/v_3/3,
                                          {u}/w_1/6, {u}/w_2/6, {u}/w_3/6,
                                         v_3/w_2/3} 
       \path[edge] (\source) -- node[weight] {} (\dest);
 \end{tikzpicture} \caption{}
\label{fig:tw3md7-conf-3}
\end{subfigure}
\caption{$\deg_G(w_i) = 3$ for $1 \leq i \leq 3$ and $\deg_G({u}) = 6$.}
\label{fig:tw3md7-conf}
\end{figure}

 \begin{lemma}
 \label{lem:conf-tw3md7}
  Let $G$ be a graph of tree-width at most $3$ and $$\deg(v) + \deg(w) \geq \max(7, \deg(v),\deg(w))+2$$ 
 for each edge $vw\in E(G)$.
 Then $G$ has a $(3,7)$-substructure $(V,W,{u})$ and one of the following holds:
\begin{enumerate}[label=\emph{\roman*})]
   \item \label{itm:conf-tw3md7-1}$|W| \geq 4$ or
   \item \label{itm:conf-tw3md7-2}$W = \{w_1, w_2, w_3\}$ with $\deg(w_i) = 3$ for each $1\leq i \leq 3$, $\deg({u}) = 6$ and $G$ has one
         of the subgraphs shown in Figure~\ref{fig:tw3md7-conf}. 
  \end{enumerate}
 \end{lemma}
 \begin{proof}
  By the assumptions $G$ has a $(3,7)$-substructure $(V,W,{u})$ as stated in Lemma~\ref{lem:tree-width-k-configurations}. 
  We will assume that $|W| \leq 3$ and hence $|W| = 3$ by Definition~\ref{def:substructure}(\ref{itm:7}). We have
  $\deg({u}) \leq |(W \cup V) \setminus \{{u}\}| = 6$ as $W$ and $V$ are disjoint, which yields $\deg({u}) = 6$ 
  by Definition~\ref{def:substructure}(\ref{itm:4} ).
  So for any $w \in W$ it holds that $$\deg(w) \geq \max(7,\deg({u}),\deg(w)) +2 - \deg({u})\geq 9 - \deg({u}) = 3 $$ and thus $\deg(w) = 3$ by 
   Definition~\ref{def:substructure}(\ref{itm:3}). 
  Either all vertices of $W$ have the same neighbourhood, or exactly two of them have the same neighbourhood, or the neighbourhoods are pairwise distinct.
 So by symmetry $G$ has one of the subgraphs 
  shown in Figure~\ref{fig:tw3md7-conf}. Note that the case $|V| = 2$ can not occur and the case $|V| = 3$ is covered in Figure~\ref{fig:tw3md7-conf-1}.
 \end{proof}

 Now we will handle the substructures for the Lemma~\ref{lem:path-width-3-delta-6}, which will be used
 to proof Theorem~\ref{thm:simple-results}(\ref{itm:simple-results-2}).

\begin{figure}
\centering
\tikzstyle{vertex}=[circle,draw, minimum size=14pt,inner sep=0pt]
\tikzstyle{edge} = [draw,-]
\tikzstyle{weight} = [font=\small]
 \begin{tikzpicture}[scale=0.8]
    \foreach \pos/\name in {{(1,3)/v_1}, {(1,2)/{u}}, {(2,2)/w_2},
                            {(0,2)/w_1},{(1,1)/v_2}, {(2,3)/{v'}}}
       \node[vertex] (\name) at \pos {$\name$};
    \foreach \source/ \dest /\weight in {v_1/{u}/5, v_1/w_1/3, v_1/w_2/3, {u}/w_1/2
                                        , {u}/w_2/2, v_2/{u}/5, v_2/w_1/3, v_2/w_2/3, {u}/{v'}/1}
       \path[edge] (\source) -- node[weight] {} (\dest);

    \foreach \vertex  in {}
        \path node[selected vertex] at (\vertex) {$\vertex$};

\end{tikzpicture}

\caption{$\deg_G(w_1) = \deg_G(w_1) = 3$ and $\deg_G({u}) = 5$.}
\label{fig:pw3md6-conf-2}
\end{figure}
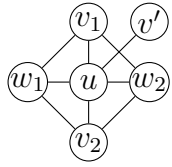

\begin{lemma}
\label{lem:conf-pw3md6}
 Let $G$ be a graph of path-width at most $3$ and $$\deg(v) + \deg(w) \geq \max(6, \deg(v),\deg(w))+2$$ 
 for each edge $vw\in E(G)$.
 Then $G$ has a $(3,6)$-substructure $(V,W,{u})$ and one of the following holds:
\begin{enumerate}[label=\emph{\roman*})]
  \item \label{itm:conf-pw3md6-1} $|W| \geq 3$ or
  \item \label{itm:conf-pw3md6-2} $|W| = 2,$ $\deg(w) = 3$ for each $w \in W,$ $\deg({u}) = 5$ and $G$ has a subgraph
         as shown in Figure~\ref{fig:pw3md6-conf-2}.
\end{enumerate}
\end{lemma}
 \begin{proof}
  By the assumptions $G$ has a $(3,6)$-substructure $(V,W,{u})$ as stated in Lemma~\ref{lem:tree-width-k-configurations}.
  There is a vertex in $V$ to
  which no element of $W$ is connected and therefore ${|N(W)| \leq 3}$.
  We will assume that ${|W| \leq 2}$ thus $|W| = 2$ by Definition~\ref{def:substructure}(\ref{itm:7}).
  As $W$ and $V$ are disjoint, ${\deg({u}) \leq |(W \cup V) \setminus \{{u}\}| = 5}$ which yields $\deg({u}) = 5$ 
  by Definition~\ref{def:substructure}(\ref{itm:4}).
  So for any $w \in W$ it holds that $$\deg(w) \geq \max(6,\deg({u}),\deg(w)) +2 - \deg({u}) \geq 8 - \deg({u}) = 3.$$ As $|N(W)| \leq 3,$
  the elements of $W$ share the same neighbourhood which results in the subgraph 
  shown in Figure~\ref{fig:pw3md6-conf-2}.
 \end{proof}

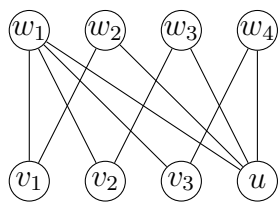
\begin{figure}
\centering
\tikzstyle{vertex}=[circle,draw,minimum size=15pt,inner sep=0pt]
\tikzstyle{edge} = [draw,-]
\tikzstyle{weight} = [font=\small]
  \begin{tikzpicture}[scale=1]\centering
    \foreach \pos/\name in {{(0,0)/v_1}, {(1,0)/v_2}, {(2,0)/v_3},{(3,0)/{u}},
                            {(0,2)/w_1}, {(1,2)/w_2}, {(2,2)/w_3},{(3,2)/w_4}}
       \node[vertex, align=center] (\name) at \pos {$\name$};
    \foreach \source/ \dest /\weight in {v_1/w_1/2, v_1/w_2/2, 
                                         v_2/w_1/2, v_2/w_3/2, 
                                         v_3/w_1/2, v_3/w_4/2,
                                         {u}/w_2/3, {u}/w_3/3,{u}/w_4/3, {u}/w_1/3}
       \path[edge] (\source) -- node[weight] {} (\dest);
 \end{tikzpicture}
\caption{$\deg_G(w_1) = 4$ and $\deg_G(w_i) = 2$ for $2 \leq i \leq 4$.}
\label{fig:pw4-4222}
\end{figure}

  Finally, here are the substructures for the Lemma~\ref{lem:pw4md11}, which implies Theorem~\ref{thm:simple-results}(\ref{itm:simple-results-3}). 

\begin{lemma}
 \label{lem:conf-pw4md10}
 Let $G$ be a graph of path-width at most 4 and $$\deg(v) + \deg(w) \geq \max(10, \deg(v),\deg(w))+2$$ 
 for each edge $vw\in E(G)$.
  Then $G$ has a $(4,10)$-substructure $(V,W,{u})$ and there is a subset $W' \subset W$ of size 4
  for which $|N(W')| \leq 4$ and one of the following holds:
\begin{enumerate}[label=\emph{\roman*})]  
  \item  \label{itm:pw4-1} Each vertex of $W'$ has a degree of at least 3,
  \item  \label{itm:pw4-2} each vertex of $W'$ has a degree of at most 3,
  \item  \label{itm:pw4-3} there are two vertices of degree 2 in $W'$ with the same neighbourhood or
  \item \label{itm:pw4-4}  $G$ has a subgraph as shown in Figure~\ref{fig:pw4-4222}.
  \end{enumerate} 
\end{lemma}
\begin{proof}
 By the assumptions $G$ has a $(4,10)$-substructure $(V,W,{u})$ as stated in Lemma~\ref{lem:tree-width-k-configurations}.
 There is one vertex
 $v' \in V$ to which none of the vertices of $W$ are connected. So we have $|N(W)| \leq 4$ and $|W| \geq 4$ by Definition~\ref{def:substructure}(\ref{itm:7}). 
 We will assume that~\ref{itm:pw4-1}–\ref{itm:pw4-3} do not hold. 
 Because~\ref{itm:pw4-1} does not hold and $|N(W)| \leq 4$, there is at least one vertex $w_4 \in W$ of degree 2. 
 The hypothesis of the lemma implies $deg(w) \geq 2$ for
 every vertex $w$. This yields that 
 $$\deg({u}) \geq \max(10,\deg({u}),\deg(w_4)) +2 - \deg(w_4) \geq 12 - \deg(w_4) = 10.$$ As $V$ and $W$ are disjoint, $|V| = 5$ and 
 by  Definition~\ref{def:substructure}(\ref{itm:8}) we have
 $|W| \geq |N({u}) \setminus (V \setminus \{{u}\})| \geq \deg({u}) - 4 \geq 6$. On the one hand, as~\ref{itm:pw4-2} does not 
 hold there is at least one
 vertex $w_1 \in W$ of degree 4. On the other hand, as~\ref{itm:pw4-1} does not hold and $|W| \geq 6,$
  there are two more vertices $w_2$ and $w_3 \in W\setminus \{w_1,w_4\}$
 of degree 2. As $w_2,$ $w_3$ and $w_4 \in N({u})$ and~\ref{itm:pw4-3} does not hold, we get the subgraph shown in Figure~\ref{fig:pw4-4222}.
\end{proof}
 
\subsection{{\emph{K}}-trees}
 A $k$-tree is a graph that can be constructed from a complete graph on $k+1$ vertices by 
 iteratively \emph{adding}
 a new vertex and connecting it to all vertices of a complete subgraph of order $k$. It is easy to see that a $k$-tree has
 tree-width $k$. In fact each graph of tree-width $k$ is a subgraph of some $k$-tree. So as said in the introduction, $k$-trees are edge-maximal
 graphs of tree-width $k$ (see \cite{opac-b1131120} for details).
 The following lemma characterises 3-trees of maximum degree at most
 6. We will colour these in the next section and use this with Theorem~\ref{thm:simple-results} to prove Theorem~\ref{thm:3-trees} 
 in Section~\ref{sec:results}.

\begin{figure}
\centering
\tikzstyle{vertex}=[circle,draw,inner sep=0pt,font=\small]
\tikzstyle{edge} = [draw,-]
\tikzstyle{weight} = [font=\small]

\begin{subfigure}[b]{0.4\textwidth}\centering
 \begin{tikzpicture}[scale=0.5]

    \foreach \pos/\name in {{(1,1)/v_3}, {(-2,-1)/v_2}, {(4,-1)/v_4},
                            {(1,4)/v_1}}
       \node[vertex, align=center] (\name) at \pos {$\name$};
    \foreach \source/ \dest /\weight in {v_1/v_2/2, v_1/v_3/3, v_1/v_4/3,
                                         v_2/v_3/3, v_2/v_4/3,
                                         v_3/v_4/2}
       \path[edge] (\source) -- node[weight] {} (\dest);

    \foreach \vertex  in {}
        \path node[selected vertex] at (\vertex) {};

 \end{tikzpicture}
 \caption{}
\label{fig:small-3-trees-1}
\end{subfigure}
\begin{subfigure}[b]{0.4\textwidth}\centering
 \begin{tikzpicture}[scale=0.5]

    \foreach \pos/\name in {{(1,1)/v_3}, {(-2,-1)/v_2}, {(4,-1)/v_4},
                            {(1,4)/v_1},{(0,1)/v_5}}
       \node[vertex, align=center] (\name) at \pos {$\name$};
    \foreach \source/ \dest /\weight in {v_1/v_2/2, v_1/v_3/3, v_1/v_4/3, v_1/v_5/2, 
                                         v_2/v_3/3, v_2/v_4/3, v_2/v_5/3,
                                         v_3/v_4/2, v_3/v_5/2}
       \path[edge] (\source) -- node[weight] {} (\dest);

    \foreach \vertex  in {}
        \path node[selected vertex] at (\vertex) {};

 \end{tikzpicture}

\caption{}
\label{fig:small-3-trees-2}
\end{subfigure}
\begin{subfigure}[b]{0.4\textwidth}\centering
 \begin{tikzpicture}[scale=0.5]

    \foreach \pos/\name in {{(1,1)/v_3}, {(-2,-1)/v_2}, {(4,-1)/v_4},
                            {(1,4)/v_1},{(-3,1)/v_5},{(-1,4)/v_6}}
       \node[vertex, align=center] (\name) at \pos {$\name$};
    \foreach \source/ \dest /\weight in {v_1/v_2/2, v_1/v_3/3, v_1/v_4/3, v_1/v_5/2, v_1/v_6/3,
                                         v_2/v_3/3, v_2/v_4/3, v_2/v_5/3,
                                         v_3/v_4/2, v_3/v_5/2, v_3/v_6/3,
                                         v_2/v_6/3} 
       \path[edge] (\source) -- node[weight] {} (\dest);

    \foreach \vertex  in {}
        \path node[selected vertex] at (\vertex) {};

 \end{tikzpicture}\caption{}
\label{fig:small-3-trees-3}
\end{subfigure}
\begin{subfigure}[b]{0.4\textwidth}\centering
 \begin{tikzpicture}[scale=0.5]

    \foreach \pos/\name in {{(1,1)/v_3}, {(-2,-1)/v_2}, {(4,-1)/v_4},
                            {(1,4)/v_1},{(0,1)/v_5},{(2,1)/v_6}}
       \node[vertex, align=center] (\name) at \pos {$\name$};
    \foreach \source/ \dest /\weight in {v_1/v_2/2, v_1/v_3/3, v_1/v_4/3, v_1/v_5/2, v_1/v_6/3,
                                         v_2/v_3/3, v_2/v_4/3, v_2/v_5/3,
                                         v_3/v_4/2, v_3/v_5/2, v_3/v_6/3,
                                         v_4/v_6/3} 
       \path[edge] (\source) -- node[weight] {} (\dest);

    \foreach \vertex  in {}
        \path node[selected vertex] at (\vertex) {};

 \end{tikzpicture}\caption{}
\label{fig:small-3-trees-4}
\end{subfigure}
\begin{subfigure}[b]{0.4\textwidth}\centering
 \begin{tikzpicture}[scale=0.5]

    \foreach \pos/\name in {{(1,1)/v_3}, {(-2,-1)/v_2}, {(4,-1)/v_4},
                            {(1,4)/v_1},{(-3,1)/v_5},{(-1,4)/v_6}, {(2,1)/v_7}}
       \node[vertex, align=center] (\name) at \pos {$\name$};
    \foreach \source/ \dest /\weight in {v_1/v_2/2, v_1/v_3/3, v_1/v_4/3, v_1/v_5/2, v_1/v_6/3,v_1/v_7/3,
                                         v_2/v_3/3, v_2/v_4/3, v_2/v_5/3,
                                         v_3/v_4/2, v_3/v_5/2, v_3/v_6/3,v_3/v_7/3,
                                         v_2/v_6/3,v_4/v_7/3} 
       \path[edge] (\source) -- node[weight] {} (\dest);

    \foreach \vertex  in {}
        \path node[selected vertex] at (\vertex) {};

 \end{tikzpicture}\caption{}
\label{fig:small-3-trees-5}
\end{subfigure}
\begin{subfigure}[b]{0.4\textwidth}\centering
 \begin{tikzpicture}[scale=0.5]

    \foreach \pos/\name in {{(1,1)/v_3}, {(-2,-1)/v_2}, {(4,-1)/v_4},
                            {(1,4)/v_1},{(-3,1)/v_5},{(-1,4)/v_6}, {(-2,3)/v_7}}
       \node[vertex, align=center] (\name) at \pos {$\name$};
    \foreach \source/ \dest /\weight in {v_1/v_2/2, v_1/v_3/3, v_1/v_4/3, v_1/v_5/2, v_1/v_6/3,v_1/v_7/3,
                                         v_2/v_3/3, v_2/v_4/3, v_2/v_5/3,v_2/v_7/3,
                                         v_3/v_4/2, v_3/v_5/2, v_3/v_6/3,v_3/v_7/3,
                                         v_2/v_6/3} 
       \path[edge] (\source) -- node[weight] {} (\dest);

    \foreach \vertex  in {}
        \path node[selected vertex] at (\vertex) {};

 \end{tikzpicture}\caption{}
\label{fig:small-3-trees-6}
\end{subfigure}
\begin{subfigure}[b]{0.4\textwidth}\centering
 \begin{tikzpicture}[scale=0.5]

    \foreach \pos/\name in {{(1,1)/v_3}, {(-2,-1)/v_2}, {(4,-1)/v_4},
                            {(1,4)/v_1},{(0,1)/v_5},{(2,1)/v_6},{(1,0)/v_7}}
       \node[vertex, align=center] (\name) at \pos {$\name$};
    \foreach \source/ \dest /\weight in {v_1/v_2/2, v_1/v_3/3, v_1/v_4/3, v_1/v_5/2, v_1/v_6/3,
                                         v_2/v_3/3, v_2/v_4/3, v_2/v_5/3, v_2/v_7/3,
                                         v_3/v_4/2, v_3/v_5/2, v_3/v_6/3,v_3/v_7/3,
                                         v_4/v_6/3, v_4/v_7/3} 
       \path[edge] (\source) -- node[weight] {} (\dest);

    \foreach \vertex  in {}
        \path node[selected vertex] at (\vertex) {};

 \end{tikzpicture}\caption{}
\label{fig:small-3-trees-7}
\end{subfigure}
\begin{subfigure}[b]{0.4\textwidth}\centering
 \begin{tikzpicture}[scale=0.5]
    \foreach \pos/\name in {{(1,1)/v_3}, {(-2,-1)/v_2}, {(4,-1)/v_4},
                            {(1,4)/v_1},{(0,1)/v_5},{(2,1)/v_6},{(3,4)/v_8}, {(1,0)/v_7}}
       \node[vertex, align=center] (\name) at \pos {$\name$};
    \foreach \source/ \dest /\weight in {v_1/v_2/2, v_1/v_3/3, v_1/v_4/3, v_1/v_5/2, v_1/v_6/3, v_1/v_8/3,
                                         v_2/v_3/3, v_2/v_4/3, v_2/v_5/3,v_2/v_7/3,
                                         v_3/v_4/2, v_3/v_5/2, v_3/v_6/3, v_6/v_8/3,v_3/v_7/3,
                                         v_4/v_6/3, v_4/v_8/3, v_4/v_7/3} 
       \path[edge] (\source) -- node[weight] {} (\dest);

    \foreach \vertex  in {}
        \path node[selected vertex] at (\vertex) {};

 \end{tikzpicture}\caption{}
\label{fig:small-3-trees-8}
\end{subfigure}
\begin{subfigure}[b]{0.4\textwidth}\centering
 \begin{tikzpicture}[scale=0.5]

    \foreach \pos/\name in {{(1,1)/v_3}, {(-2,-1)/v_2}, {(4,-1)/v_4},
                            {(1,4)/v_1},{(0,1)/v_5},{(2,1)/v_6}, {(1,0)/v_7}, {(1,5)/v_8}}
       \node[vertex, align=center] (\name) at \pos {$\name$};
    \foreach \source/ \dest /\weight in {v_1/v_2/2, v_1/v_3/3, v_1/v_4/3, v_1/v_5/2, v_1/v_6/3, v_1/v_8/3,
                                         v_2/v_3/3, v_2/v_4/3, v_2/v_5/3, v_2/v_7/3, v_2/v_8/3,
                                         v_3/v_4/2, v_3/v_5/2, v_3/v_6/3, v_3/v_7/3, v_4/v_8/3, 
                                         v_4/v_6/3, v_4/v_7/3} 
       \path[edge] (\source) -- node[weight] {} (\dest);

    \foreach \vertex  in {}
        \path node[selected vertex] at (\vertex) {};

 \end{tikzpicture}\caption{}
\label{fig:small-3-trees-9}
\end{subfigure}
\caption{Some  3-trees}
\label{fig:small-3-trees}
\end{figure}
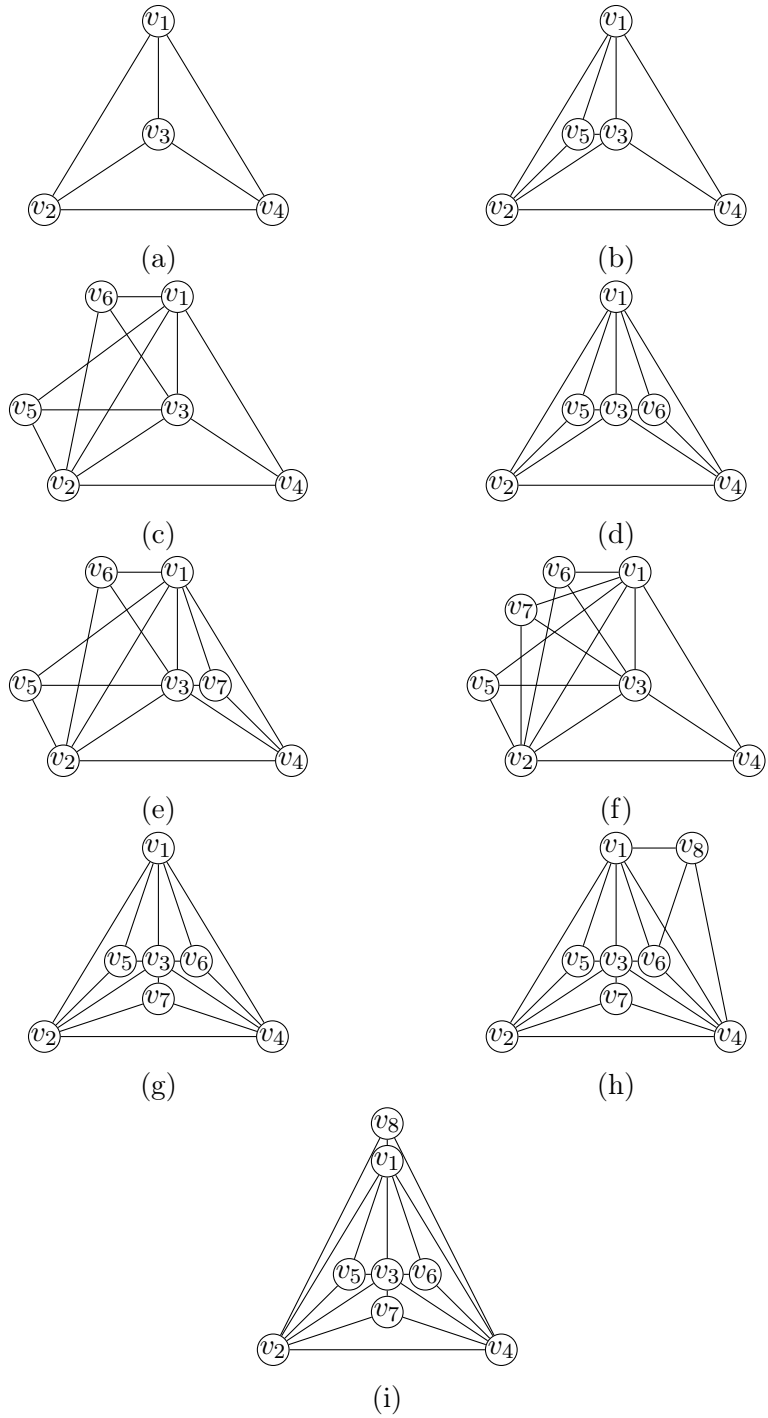

\begin{lemma}
\label{lem:small-3-trees}
 Let $G$ be a 3-tree of maximum degree at most 6. Then $G$ has path-width 3 or is isomorphic to one of 
 the graphs shown in Figure~\ref{fig:small-3-trees}.
\end{lemma}
\begin{proof}
 We will construct 3-trees and see that the maximum degree grows quickly larger than 6. 
 We start with the $K_4$ in Figure~\ref{fig:small-3-trees-1}. If we add a vertex, we get the graph shown in Figure
~\ref{fig:small-3-trees-2}. Adding another vertex yields the graph in Figure~\ref{fig:small-3-trees-3} 
 or~\ref{fig:small-3-trees-4} by symmetry. As any triangle in the last two graphs contains a vertex of degree 5, adding any other
 vertex will raise the maximum degree to 6. Therefore the graphs shown in the Figures~\ref{fig:small-3-trees-1}--\ref{fig:small-3-trees-4} 
 are exactly the 3-trees of maximum degree at most 5. Now let $G$ be 3-tree of maximum degree 6 with a width 3 tree decomposition $(T,\mathcal{V})$,
 such that each bag has exactly four vertices and for all $tt' \in E(T)$ we have $|V_{t} \cap V_{tt'}| = 3$. 
 If $T$ is a path, we are done. Consequently 
 let $t \in V(T)$ be a vertex with neighbours $t_1,$ $t_2$ and 
 $t_3 \in V(T)$. Set $X = V_t \cup V_{t_1} \cup V_{t_2} \cup V_{t_3}$. 
 Then $|V_{t_1} \cap V_{t_2} \cap V_{t_3}| = 1,$ $2$ or $3$. In
 these cases the graph induced by $X,$ $G[X],$ is isomorphic to the graph shown in Figure~\ref{fig:small-3-trees-7},~\ref{fig:small-3-trees-5} 
 or~\ref{fig:small-3-trees-6}. 
 Observe that any triangle in the graphs of Figure~\ref{fig:small-3-trees-5} and~\ref{fig:small-3-trees-6}
 has at least one vertex of degree 6, which yields $G[X] = G$  as $\Delta(G) = 6$. 
 So let $G[X]$ be the graph shown in Figure~\ref{fig:small-3-trees-7} and assume that there is at least one more vertex 
 $v_8 \in V(G) \setminus X$. 
 Up to symmetry there are only two triangles in $G[X],$ which do not already contain a vertex of degree 6 and 
 to whose vertices another vertex $v_8$ can be
 connected without raising the maximum degree. This results in one of the graphs shown in Figures~\ref{fig:small-3-trees-8} and~\ref{fig:small-3-trees-9}. As 
 all triangles in these graphs contain at least one vertex of degree 6, the graphs 
 of~Figure~\ref{fig:small-3-trees-5}--\ref{fig:small-3-trees-9} cover all 3-trees of maximum degree at most 6 that may not have a width $3$ path decomposition. 
\end{proof}

\section{Colouring substructures}
\label{sec:instances}
 In this section we will solve the instances of list edge-colouring related to the substructures we have just found.
 We will generally assume that the size of a list is exactly the size of its respective lower bound. 
 We can always try to \emph{colour $G$ semi-greedily}, by iteratively colouring 
 an edge with a smallest list of remaining colours with an arbitrary available colour. 
  The following result has already been mentioned. We will apply it several times.
 \begin{theorem}[Galvin, 1994]
 \label{thm:galvin}
 Let $G$ be a bipartite graph; then ${\chi'(G)=\ch'(G)}$. 
 \end{theorem}
 For a graph $G$ with an assignment of lists $L$ to the edges of $G$ and ${e,f \in E(G)}$ we call two colours $c_1 \in L(e)$
 and $c_2 \in L(f)$ \emph{compatible} if $c_1 = c_2$ or for each edge $g$ that is adjacent to both $e$ and $f$ the list $L(g)$ contains
 at most one of the two colours $c_1$ and $c_2$.
 The following lemma turns out to be quite useful in order to solve instances of list edge-colouring with small graphs. The idea of the proof can be extracted from 
 \cite{Cariolaro_theedge-choosability}.
 \begin{lemma}
 \label{lem:cariolaro}
 Let $G$ be a graph with an assignment of lists $L$ to the edges of $G$ 
 and let $v_1v_2,$ $w_1w_2 \in E(G)$ be two non-adjacent edges. If 
  \begin{equation} \label{equ:cariolaro}
  |L({v_1v_2})||L({w_1w_2})| > \sum_{v_iw_j \in E(G)} \lfloor\frac{|L({v_iw_j})|}{2}\rfloor  \lceil\frac{|L({v_iw_j})|}{2}\rceil,
 \end{equation}
 then there are
  compatible colours $c_1 \in L({v_1v_2})$ and $c_2 \in L({w_1w_2})$.
\end{lemma}
 \begin{proof}
  If the lists
  of the edges $v_1v_2$ and $w_1w_2$ share a colour $c$ we are done.
  Therefore assume that $L({v_1v_2}) \cap L({w_1w_2}) = \emptyset$. This yields that there are  $|L({v_1v_2})||L({w_1w_2})|$ pairs of
  distinct colours $(c,c)$ with $c \in L({v_1v_2})$ and $c' \in L({w_1w_2})$. But an edge $v_iw_j \in E(G)$ 
  for $1\leq i,j\leq 2$ can contain both colours of at most
  $\lfloor\frac{|L({v_iw_j})|}{2}\rfloor  \lceil\frac{|L({v_iw_j})|}{2}\rceil$ of those pairs.
  So if (\ref{equ:cariolaro}) holds the desired compatible colours
  $c_1 \in L({v_1v_2})$ and $c_2 \in L({w_1w_2})$ exist.
 \end{proof}
 Remark that (\ref{equ:cariolaro}) holds, if all involved lists have a size of exactly $k,$ where $k$ is an odd number. It also holds if all
 involved lists have the same size and at least one of the four edges $v_iw_j$ is missing from $G.$
 
\subsection{Small 3-trees}
 We now analyse the list chromatic index of the 3-trees, which we have found in Lemma~\ref{lem:small-3-trees}.
 \begin{lemma}
 \label{lem:small_graphs_of_tree-width_three}
 Let $G$ be one of the graphs shown in Figure~\ref{fig:small-3-trees}. Then ${\ch'(G) = \chi'(G)}.$
\end{lemma}

\begin{proof}
 Let $G_a$ be the graph shown in Figure~\ref{fig:small-3-trees-1} with an assignment of lists $L_a$ to the edges of $G_a,$
 where each edge has list of size at least $\Delta(G_a) = 3.$ For two non-adjacent edges $e,$ $f \in G_a$, 
 use Lemma~\ref{lem:cariolaro} to pick two compatible colours $c_1 \in L_1(e)$ and $c_2 \in L_1(f)$
 and colour $e$ and $f$ with them. The rest of the graph forms a $K_{2,2}$ whose edges retain enough available colours to 
 apply Theorem~\ref{thm:galvin}.
 
 Let $G_b$ be the graph shown in Figure~\ref{fig:small-3-trees-2}.
 As it has 5 vertices, out of 3 edges at least 2 are adjacent. Since the total number of edges is 9, we have $\chi'(G_2) = \Delta +1 = 5.$ 
 So for a given assignment of lists $L_b$ to the edges of $G_b,$ where each list has a size of at least 5, pick two compatible colours 
 $c_1 \in L_b(v_1v_2)$ and
 $c_2 \in L_b(v_3v_5)$ with Lemma~\ref{lem:cariolaro} and colour these edges with them. We apply Lemma~\ref{lem:cariolaro} another time to
  pick two compatible and available colours $c_3 \in L_b(v_1v_3)$ and $c_4 \in L_b(v_2v_4)$ and colour these edges with them. 
 The rest of the
 graph can be coloured semi-greedily. 
 
 Let $G_c$ be the graph shown in Figure~\ref{fig:small-3-trees-3} with lists of colours $L_c,$ each of size at least $\Delta(G_c) = 5$,
 assigned to its edges. Colour the triangle $v_1,$ $v_2,$ $v_3$ semi-greedily and apply Theorem~\ref{thm:galvin} to the rest.
 
  Let $G_d$ be the graph shown in Figure~\ref{fig:small-3-trees-4} with lists $L_d,$ each of size at least $\Delta(G_d) = 5,$  
  assigned to its edges. Use Lemma~\ref{lem:cariolaro} to pick two compatible colours
 $c_1 \in L_d(v_1v_3)$ and $c_2 \in L_d(v_2v_4)$ and colour the respective edges with them. Apply Lemma~\ref{lem:cariolaro} again to pick
 two compatible and available colours $c_3 \in L_d(v_1v_2)$ and $c_4 \in L_d(v_3v_5).$ Use Lemma~\ref{lem:cariolaro} a last time to pick
 two compatible and available colours ${c_5 \in L_d(v_1v_4)}$ and $c_6 \in L_d(v_3v_6),$ colour the respective edges with these colours 
 and finish semi-greedily.
 
  Let $G$ be one of the graphs shown in Figure~\ref{fig:small-3-trees-5},~\ref{fig:small-3-trees-6},~\ref{fig:small-3-trees-7} 
  or~\ref{fig:small-3-trees-9} with lists $L,$ each of size at least $\Delta(G) = 6,$ 
  assigned to its edges. The vertices $v_1,$ $v_2,$ $v_3$ and $v_4$ induce a $K_4$ in $G.$ Colour its edges as shown in
  the case of Figure~\ref{fig:small-3-trees-1} and observe
  that the rest of the graph forms a bipartite graph whose edges are adjacent to already coloured edges on at most one end. 
  So the lists of remaining colours retain sizes big enough
  to apply Theorem~\ref{thm:galvin}. For the graph $G_h$ shown in Figure~\ref{fig:small-3-trees-8}, we can apply
  the same argument to colour $G_h - v_6v_8$ and then finish semi-greedily afterwards.
\end{proof}

\subsection{Bipartite graphs}
\label{sec:intances-bip}
In the following we will discuss some instances of list edge-colouring with bipartite graphs. Throughout this subsection
$G$ will be a bipartite graph with biparts $V,~W$ of the same size. Vertices denoted as $v,~v'$ etc. will be assumed to
lie in $V$, vertices denoted as $w,~w'$ will lie in $W$.
 We will apply the following refined version of Theorem~\ref{thm:galvin}. It can be found 
 in~\cite{journals/cpc/Slivnik96} and~\cite{journals/jct/BorodinKW97}.

 \begin{theorem}
 \label{thm:slivnik}
 Let $G$ be a bipartite graph with an edge-colouring $\mathcal{C}$
 and an assignment of lists $L$ to the edges of $G$ such that 
\begin{alignat}{3}
&|L({w'v'})|&&\geq|\{ wv' \in E(G) \text{ }; \text{ } \mathcal{C}(wv') > \mathcal{C}(w'v')  \}| \nonumber \\
& &&+ |\{ w'v \in E(G) \text{ }; \text{ }\mathcal{C}(w'v) < \mathcal{C}(w'v')\}| + 1 \nonumber
\end{alignat}
 for each edge $w'v' \in E(G).$ Then there is an $L$-edge-colouring of $G.$
 \end{theorem}
 
 As a first application we get the following lemma. We will use it later in the proof of Lemma~\ref{lem:cherry-configuration-two-max-deg-7}.

 \begin{figure}
\centering
\tikzstyle{vertex}=[circle,draw,minimum size=14pt,inner sep=0pt]
\tikzstyle{edge} = [draw,-]
\tikzstyle{weight} = [font=\small,draw,fill           = white,
                                  text           = black]

\begin{subfigure}[b]{0.45\textwidth}\centering
 \begin{tikzpicture}[scale=1.1]
    \foreach \pos/\name in {{(1,1)/{u}}, {(-1,-1)/v_1}, {(3,-1)/v_3},
                            {(1,3)/v_2},{(-0.2,1)/w_1},{(2.2,1)/w_2},{(1,-0.2)/w_3}}
       \node[vertex, align=center] (\name) at \pos {$\name$};
    \foreach \source/ \dest /\weight in {v_2/w_1/3, v_2/w_2/2, v_3/w_3/2,
                                          v_1/w_1/2, v_1/w_3/2,
                                          {u}/w_1/4, {u}/w_2/4, {u}/w_3/3,
                                         v_3/w_2/2} 
       \path[edge] (\source) -- node[weight] {$\weight$} (\dest);
 \end{tikzpicture} 
 \caption{The integers indicate the minimum sizes of the respective lists.}
 \label{fig:cherry-configuration-bipartite-list-size}
\end{subfigure}
\begin{subfigure}[b]{0.45\textwidth}\centering
 \begin{tikzpicture}[scale=1.1]
    \foreach \pos/\name in {{(1,1)/u}, {(-1,-1)/v_1}, {(3,-1)/v_3},
                            {(1,3)/v_2},{(-0.2,1)/w_1},{(2.2,1)/w_2},{(1,-0.2)/w_3}}
       \node[vertex, align=center] (\name) at \pos {$\name$};
    \foreach \source/ \dest /\weight in {v_2/w_1/3, v_2/w_2/1, v_3/w_3/1,
                                          v_1/w_1/1, v_1/w_3/2,
                                          u/w_1/2, u/w_2/3, u/w_3/4,
                                         v_3/w_2/2} 
       \path[edge] (\source) -- node[weight] {$\weight$} (\dest);
 \end{tikzpicture} 
 \caption{The integers indicate the colours given to the edges}
 \label{fig:cherry-configuration-bipartite-proper colouring}
\end{subfigure}
\caption{}
\label{fig:cherry-configuration-bipartite}
\end{figure}

\begin{lemma}
\label{lem:cherry-configuration-bipartite}
 Let $G$ be the bipartite graph shown in Figure~\ref{fig:cherry-configuration-bipartite-list-size} with lists of colours assigned to 
 the edges, where the minimum sizes of 
 the lists are indicated by the integers on the edges. Then there is an
 $L$-edge-colouring of $G$. 
\end{lemma}

 \begin{proof}
  Use Theorem~\ref{thm:slivnik} with the edge-colouring in Figure~\ref{fig:cherry-configuration-bipartite-proper colouring}.
 \end{proof}
 We call an assignment of lists $L$ to the edges of a bipartite graph $G$ \emph{$V$-dominated} if 
 $|L({wv})| \geq \deg(v)$ for each $v \in V.$ We say that $G$ is \emph{$V$-choosable}, if every $V$-dominated assignment of lists  $L$
 permits an $L$-edge-colouring. The next lemma can also be found in \cite{journals/jct/BorodinKW97}.

 \begin{lemma}
 \label{lem:Y-respecting}
  Let $G$ be a bipartite graph with an assignment of lists $L$ to the edges of $G$ such that 
  $L(wv) = \{1, \ldots, \deg(v) \}$ for $w \in W$ and $v \in V.$ Then $G$ is $V$-choosable if and only if it has an $L$-edge-colouring.
 \end{lemma}


\begin{lemma}
 \label{lem:2-regular}
  Let $G$ be a bipartite graph in which $\deg(w) \leq 3$ for each $w \in W.$ Then $G$ is $V$-choosable if
  it has a 2-regular spanning subgraph~$H.$ 
 \end{lemma}

\begin{proof}
 By Lemma~\ref{lem:Y-respecting} it suffices to show that $G$ has an $L$-edge-colouring from the lists $L(wv) = \{1, \ldots, \deg(v)\}$ for $w \in W$ and $v \in V.$
 By Theorem~\ref{thm:galvin} we can colour the edges of the subgraph $H$ with colours 1 and 2. As the remaining edges form stars with their centres in $V$,
 we can finish semi-greedily. 
\end{proof}

\begin{lemma}
 \label{lem:K(3,3)_Y-choosable} 
  Let $G$ be a bipartite graph with biparts $V$ and $W$, of size $|V| \leq |W| = 3$. Then we can
  find a subset $W' \subset W$ such that the graph $G[W' \cup N(W')]$ is $N(W')$-choosable.
\end{lemma}

\begin{proof}
 We can assume that each $w \in W$ has
 $\deg(w) \geq 2.$ Otherwise $\{w\}$ would be $N(\{w\})$-choosable. If there is a vertex $v \in V$ of degree $1,$ say $vw \in W,$ we set $W' = W \setminus \{w\}$ 
 and apply Theorem~\ref{thm:galvin}.
 So assume that $G$ has minimum degree 2. 
 If $V$ has less than 3 vertices any two vertices of $W$ will work by Theorem~\ref{thm:galvin}.
 If $G$ has 6 edges it is a cycle and we can apply Theorem~\ref{thm:galvin} again.
 If $G$ has 7 edges, there are exactly two vertices $w\in W$ and $v \in V$ of degree 3 and $wv \in E(G).$ As $G - wv$ is 
 2-regular,
 we are done by Lemma~\ref{lem:2-regular}. If $G$ has 8 edges, there are four vertices $w_1,$ $w_2 \in W$ and
 $v_1,$ $v_2 \in V$ of degree 3 and $w_1v_1,$ $w_2v_2 \in E(G).$ As $G -w_1v_1-w_2v_2$ is 2-regular,
 we are done by Lemma~\ref{lem:2-regular}.
 If $G$ has 9 edges it is isomorphic to $K_{3,3}$ and we can apply Theorem~\ref{thm:galvin}.
\end{proof}

For the proof of the next lemma we define a \emph{$k$-vertex} to be a vertex of degree $k$ and a \emph{matching} to be a set of pairwise non-adjacent edges.

\begin{figure}
\centering
\tikzstyle{vertex}=[circle,draw,minimum size=14pt,inner sep=0pt]
\tikzstyle{edge} = [draw,-]
\tikzstyle{weight} = [font=\small]
 \begin{tikzpicture}[scale=1]
    \foreach \pos/\name in {{(0,0)/v_1}, {(1,0)/v_2}, {(2,0)/v_3},{(3,0)/v_4},
                            {(0,2)/w_1}, {(1,2)/w_2}, {(2,2)/w_3},{(3,2)/w_4}}
       \node[vertex, align=center] (\name) at \pos {$\name$};
    \foreach \source/ \dest /\weight in {v_1/w_1/2, v_1/w_2/2,
                                         v_2/w_1/2, v_2/w_3/2, 
                                         v_3/w_1/2, v_3/w_4/2,
                                         v_4/w_2/3, v_4/w_3/3, v_4/w_4/3}
       \path[edge] (\source) -- node[weight] {} (\dest);
 \end{tikzpicture}
\caption{}
\label{fig:K44_exception}
\end{figure}

\begin{lemma}
 \label{lem:K44-min-deg-2}
 Let $G$ be a bipartite graph of minimum degree 2 with biparts $V$ and $W$ such that $|V| = |W| = 4$ and $\deg(w) \leq 3$ for each $w\in W.$ Then $G$ is $V$-choosable, 
 except for the case where $G$ is isomorphic
 to the graph in Figure~\ref{fig:K44_exception}.
\end{lemma}

\begin{proof} 
  We will show that $G$ has a 2-regular spanning subgraph, from which the result follows immediately by 
  Lemma~\ref{lem:2-regular}. 
  Write $W = \{w_1,w_2,w_3,w_4\}$ and $V = \{v_1,v_2,v_3,v_4\}.$ If there are no 3-vertices in $W$, we are done. Say the 3-vertices of $W$ are $w_1, \ldots, w_k$. We first
  assume that $k=1$ and the other 3-vertex is $v_4.$ If $w_1$ is
  adjacent to $v_4$ then $G-w_1v_4$ is a 2-regular spanning subgraph; otherwise, $N(w_1) = \{ v_1, v_2, v_3\},$ $N(v_4) = \{ w_2, w_3, w_4\},$ and the remaining
  three edges form a matching between $N(w_1)$ and $N(v_4),$ which results in the exceptional graph in Figure~\ref{fig:K44_exception}.
  
  Now assume that $k=2.$ If there is a 4-vertex $v_1 \in V$ then $G - w_1v_1 -w_2v_1$ is a 2-regular spanning subgraph; otherwise, we may assume that
  $v_1$ and $v_2$ are 3-vertices, and there is a matching $M$ of two edges between $\{ v_1,v_2\}$ and $\{ w_1,w_2\}$, since each of these vertices is 
  adjacent to at least one vertex in the other set; then $G-M$ is a 2-regular spanning subgraph.
  
  If $k = 3,$ let $v_i$ be a vertex of degree at least 3 that is adjacent to $w_3,$ and note that $G-w_3v_i$ satisfies the hypothesis of the lemma. 
  As in $G-w_3v_i$, $W$ has only two 3-vertices, we can find a 2-regular spanning subgraph as in the above lines. The case of $k = 4$ follows the same way.
\end{proof}

\begin{lemma}
 \label{lem:K(4,4)_almost_Y-choosable}
 Let $G$ be a bipartite graph with biparts  $V$ and $W$ such that $|V| \leq |W| = 4$ and $\deg(w) \leq 3$ for each $w\in W.$ Then we can
  find a subset $W' \subset W$ such that the graph $G[W' \cup N(W')]$ is $N(W')$-choosable, except for the case 
  where $G$ is isomorphic
 to the graph in Figure~\ref{fig:K44_exception}.
\end{lemma}
\begin{proof}
  If $V$ has a size of less than 4, Lemma~\ref{lem:K(3,3)_Y-choosable} applies and we are done. So $|V| = 4$ and 
 we may assume that for each proper subset $W' \subset W$ we have ${|W'| < |N(W')|}$
 as we could apply Lemma~\ref{lem:K(3,3)_Y-choosable} otherwise. As before there is no vertex $w \in W$ of degree 1.
 If there is a vertex
 $v \in V$ of degree 1, say $wv \in E(G),$ the result holds with $W' = W \setminus \{w\}.$
 This yields
 that $G$ has minimum degree at least 2. By Lemma~\ref{lem:K44-min-deg-2} $G$ is $V$-choosable or isomorphic to the graph in Figure~\ref{fig:K44_exception}.
\end{proof}

\begin{lemma} 
\label{lem:k44-min-deg-3}
 Let $G$ be a bipartite graph with biparts $V$ and $W$ such that $|V| \leq |W| = 4$ and $\deg(w) \geq 3$ for each $w\in W.$ Further let
 there be one vertex ${u} \in V$ to which each vertex of $W$ is connected.
 Then we can
 find a subset $W' \subset W$ such that the graph $G[W' \cup N(W')]$ is $N(W')$-choosable.                                                                                              
\end{lemma}
\begin{proof}
  If $V$ has a size of less than 4, Lemma~\ref{lem:K(3,3)_Y-choosable} applies and we are done. So $|V| = 4$ and  as before
 we may assume that for each proper subset $W' \subset W$ we have $|W'| < |N(W')|$
 as we could apply Lemma~\ref{lem:K(3,3)_Y-choosable} otherwise. This implies that $G$ has minimum degree 2 as seen in the proof of Lemma~\ref{lem:K(3,3)_Y-choosable}.
 Denote by $k$ the number 4-vertices in $W.$ We have $|E(G)| = 4k+3(4-k) =12+k$ and therefore
 there are $8+k$ edges incident to the vertices of $V \setminus \{{u}\}.$ Thus there are at least as many 4-vertices in $V$ as there are in $W.$ Let $L$ be an
 assignment of lists to the edges $vw \in E(G)$ such that $L(vw) = \{1, \ldots, \deg(v)\}.$ We will show that $G$ has an $L$-edge-colouring and is hence 
 $V$-choosable by Lemma~\ref{lem:Y-respecting}. Denote by $X$ the 4-vertices of $V$. 
 Let $Y$ be the set of $|X|$ vertices with of largest degree in $W$.
 Note that each 4-vertex
 of $W$ is included in $Y.$ As every vertex of $X$
 is connected to every vertex of $Y,$ there is a matching $M$ of size $|X|$ between $X$ and $Y.$ Colour
 the edges of $M$ with colour 4. The graph $G  -M$ retains a minimum degree of at least~2 and has a maximum
 degree of at most 3. Thus we can use Lemma~\ref{lem:K44-min-deg-2} colour the edges of $G-M$ from the lists of remaining colours.
\end{proof}

\begin{figure}
\centering
\tikzstyle{vertex}=[circle,draw,minimum size=14pt,inner sep=0pt]
\tikzstyle{edge} = [draw,-]
\tikzstyle{weight} = [font=\small]
\begin{subfigure}[b]{0.3\textwidth}\centering
 \begin{tikzpicture}[scale=1]
    \foreach \pos/\name in {{(0,0)/v_1}, {(1,0)/v_2}, {(2,0)/v_3},{(3,0)/v_4},
                            {(0,2)/w_1}, {(1,2)/w_2}, {(2,2)/w_3},{(3,2)/w_4}}
       \node[vertex, align=center] (\name) at \pos {$\name$};
    \foreach \source/ \dest /\weight in {v_1/w_1/2, v_1/w_2/2, w_1/v_4/4,
                                         v_2/w_1/2, v_2/w_3/2, 
                                         v_3/w_1/2, v_3/w_4/2,
                                         v_4/w_2/43, v_4/w_3/4,v_4/w_4/4}
       \path[edge] (\source) -- node[weight] {} (\dest);
 \end{tikzpicture}
 \caption{}
 \label{fig:pw4-ins-1}
\end{subfigure}
\begin{subfigure}[b]{0.3\textwidth}\centering
 \begin{tikzpicture}[scale=1]
    \foreach \pos/\name in {{(0,0)/q_1}, {(1,0)/q_2}, {(2,0)/q_3},{(3,0)/q_4},
                            {(0,2)/p_1}, {(1,2)/p_2}}
       \node[vertex, align=center] (\name) at \pos {$\name$};
    \foreach \source/ \dest /\weight in {q_1/p_1/2, q_1/p_2/2,
                                         q_2/p_1/2, q_2/p_2/2, 
                                         q_3/p_1/2, q_3/p_2/2,
                                         q_4/p_1/2, q_4/p_2/2}
       \path[edge] (\source) -- node[weight] {} (\dest);
 \end{tikzpicture}
 \caption{}
 \label{fig:pw4-ins-2}
\end{subfigure}

\caption{}
\label{fig:pw4-ins}
\end{figure}

\begin{lemma}
 \label{lem:pw4-ins}
 Let $G$ be the graph shown in Figure~\ref{fig:pw4-ins-1} with an assignment of lists $L$ to the edges of $G$ such that the size of
 each list $L({w_iv_j}),$ is at least $\deg(v_i)$ for $1\leq i,j \leq 4.$ 
 There is an $L$-edge-colouring of $G,$ if there is an $L^1$-colouring $\mathcal{C}$ of the
 complete bipartite graph $K_{2,4}$ 
 with bipartition
 classes ${\{p_1, p_2\}}$ and $\{q_1, q_2, q_3, q_4 \},$ from the assignment of lists $L^1,$ 
 where $L^1({p_iq_j}) = L({v_jw_1})$ for $1 \leq i \leq 2$
 and $1 \leq j \leq 4$ (see Figure~\ref{fig:pw4-ins-2}).
\end{lemma}

\begin{proof}
 If  $\mathcal{C}(p_1q_1) \notin L({\eb{1}{2}})$ colour the edges $\eb{i}{1}$ with 
 $\mathcal{C}(p_1q_i)$ for $1 \leq i \leq 4$ and finish semi-greedily. The case $\mathcal{C}(p_2q_1) \notin L({\eb{1}{2}})$
 can be handled the same way. So we can assume that $L({v_1w_1}) = L({v_{1}w_{2}})$ and by symmetry
 the same for $v_2$ and $v_3,$ hence:

 \begin{equation} \label{equ:K44_exception_1}
  L({v_iw_1}) = L({v_{i}w_{i+1}})
 \end{equation}
 for $1 \leq i \leq 3.$ 
 Now we colour the edges
\begin{itemize}
 \item $\eb{i}{1}$ with $\mathcal{C}(p_1q_i),$
 \item $\eb{4}{1}$ with $\mathcal{C}(p_1q_4)$ and
 \item $\eb{i}{i+1}$ with $\mathcal{C}(p_2q_i)$
\end{itemize}

 for $1 \leq i \leq 3.$ If we can not finish this edge-colouring, then by Lemma~\ref{lem:trident} there are colours $c_1$ and $c_2$ 
 such that

 \begin{equation}
     \label{equ:K44_exception_2}
     L({\eb{4}{i+1}}) = \{c_1,c_2,\mathcal{C}(p_1q_4), \mathcal{C}(p_2q_i)\}   
 \end{equation}
 for $1 \leq i \leq 3.$ By \eqref{equ:K44_exception_1} and \eqref{equ:K44_exception_2}, we can colour the edges as follows

\begin{itemize}
 \item $\eb{i}{1}$ with $\mathcal{C}(p_2q_i),$
 \item $\eb{4}{i+1}$ with $\mathcal{C}(p_2q_i),$
 \item $\eb{i}{i+1}$ with $\mathcal{C}(p_1q_i)$ and
 \item $\eb{4}{1}$ with $\mathcal{C}(p_2q_4)$
\end{itemize}

for $1 \leq i \leq 3$ to obtain an $L$-edge-colouring.
\end{proof}

Note that the instance shown in Figure~\ref{fig:K44_exception} can be solved in a very similar way.

\subsection{Other instances}
Now we will deal with some general instances that are related to the substructures that
 may appear within graphs of bounded tree-width, as it was shown in Section~\ref{sec:configurations}.
We start with a classic result that can be found in \cite{ErRuTa79}.
 
\begin{lemma}[Erd\H{o}s, Rubin and Taylor, 1979]
 \label{lem:cycle}
 Let $G$ be a cycle with an assignment of lists $L$ to the edges of $G,$ where the size of each list is at least 2. There is an $L$-edge-colouring of $G,$ 
 unless it is an odd cycle, all lists have size two and are identical.
\end{lemma}

\begin{lemma}
  \label{lem:trident}
  Let $G$ be the graph with vertices $V(G) = \{w_1, w_2, w_3, v\}$ and edges $E(G)= \{vw_1, vw_2, vw_3\}.$ 
   Let $L$ be an assignment of lists,
  where each list has a size of at least 2. Then there is an $L$-edge-colouring of $G$
 if either 
 \begin{itemize}
  \item one of the lists has size at least 3 or
  \item the lists are not pairwise identical.
 \end{itemize}
 \end{lemma}
\begin{proof}
 In both cases we can assume that by symmetry there is a ${c \in L({vw_1}) \setminus L({vw_2})}.$ 
 Colour ${vw_1}$ with $c$ and finish semi-greedily.
\end{proof}
\begin{lemma}
\label{lem:ballon}
 Let $G$ be a cycle with edges $e_1,\ldots,e_n$ and an additional edge $f$ that is adjacent exactly to the vertex of $C$ that
 $e_1$ and $e_n$ share. Then there is an $L$-edge-colouring of $G$ for each
 assignment of lists $L,$ where the lists have a size of at least 2 and  $|L_{e_1}| \geq 3$.
\end{lemma}
\begin{proof}
 If there is a colour $c \in L({e_n}) \setminus L({f}),$ colour $e_n$ with $c$ and finish semi-greedily. This yields $L(f) = L({e_n})$ and
 hence there is a colour $c \in L({e_1}) \setminus (L(f) \cup L({e_n}))$. Colour $e_1$ with $c$  and finish semi-greedily.
\end{proof}

\begin{lemma}
\label{lem:eight}
 Let $G$ be a graph consisting of two cycles $V$ and $W$ with edges $E(V)= \{ g_1,\ldots, g_n \}$ and 
 $E(W)= \{ f_1, \ldots, f_m \}$ respectively, that share exactly one vertex $v \in g_1 \cap g_n \cap f_1 \cap f_m.$ 
 Then for any assignment of lists $L,$ where each list has a size of at least 2 and  
 and the size of $L({g_n})$
 and $L({f_m})$ is at least 4, there is an $L$-edge-colouring of $G.$
\end{lemma}

\begin{proof}
       If there is a colour $c \in L(g_1) \setminus L(f_1),$ colour $g_1$ with $c,$ colour the edges ${g_2,~\ldots,~g_{m-1}}$ semi-greedily and 
       finish as shown in Lemma~\ref{lem:ballon}. 
       So we can assume that ${L(g_1) = L(f_1)}.$ Now suppose that there is a colour ${c \in L(g_i) \setminus L({g_{i+1}})}$ for some
       $1 \leq i \leq n-1.$ In that case we can colour $g_i$ with $c$ and finish semi-greedily. Therefore let $L(g_i) = L({g_{i+1}})$ for
       $1 \leq i \leq n-1$ and by symmetry $L({f_i}) = L({f_{i+1}})$ for  $1 \leq i \leq m-1.$ Bearing this in mind we can simply colour
       the graph semi-greedily.
\end{proof}

The next lemma will cover the instance related to the graph shown in Figure~\ref{fig:pw3md6-conf-2}. 

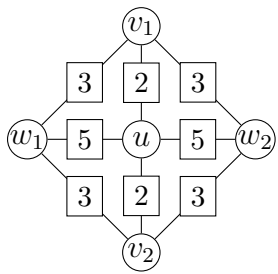
\begin{figure}
\centering
\tikzstyle{vertex}=[circle,draw,minimum size=14pt,inner sep=0pt]
\tikzstyle{edge} = [draw,-]
\tikzstyle{weight} = [font=\small,draw,fill           = white,
                                  text           = black]
 \begin{tikzpicture}[scale=1.5]
    \foreach \pos/\name in {{(1,3)/v_1}, {(1,2)/{u}}, {(2,2)/w_2},
                            {(0,2)/w_1},{(1,1)/v_2}}
       \node[vertex] (\name) at \pos {$\name$};
    \foreach \source/ \dest /\weight in {v_1/{u}/2, v_1/w_1/3,v_1/w_2/3,{u}/w_1/5
                                        , {u}/w_2/5, v_2/{u}/2  , v_2/w_1/3 , v_2/w_2/3 }
       \path[edge] (\source) -- node[weight] {$\weight$} (\dest);

    \foreach \vertex  in {}
        \path node[selected vertex] at (\vertex) {$\vertex$};

\end{tikzpicture} 
\caption{The integers indicate the minimum sizes of the respective lists.}
\label{fig:4-pyramid}
\end{figure}
\begin{lemma}
\label{lem:4-pyramid}
 Let $G$ be the graph shown in Figure~\ref{fig:4-pyramid} with lists of colours $L$ assigned to the edges, where the minimal sizes of 
the lists are indicated by the integers on the edges. Then there is an $L$-edge-colouring of $G.$ 
\end{lemma}

\begin{proof}

   If there is a colour $c \in \elgen{v_1}{{u}} \cap \elgen{v_2}{w_1},$ colour both edges with $c$ and finish semi-greedily. So we may assume
   \begin{equation}\label{equ:4-pyramid-1}
    \elgen{v_1}{{u}} \cap \elgen{v_2}{w_1}  = \emptyset.
   \end{equation}
   If there is a colour $c \in \elgen{v_2}{w_1} \cap \elgen{v_1}{w_2},$ colour both edges with $c,$ note that 
   $c \notin \elgen{v_1}{{u}}  \cup \elgen{{v_2}}{u}$
         by \eqref{equ:4-pyramid-1} (symmetry) and finish as in Lemma~\ref{lem:eight}. Therefore we have 
   \begin{equation}\label{equ:4-pyramid-2}
   \elgen{v_2}{w_1} \cap \elgen{v_1}{w_2} = \emptyset.
   \end{equation}
         If there is a colour $c \in \elgen{u}{{w_1}} \cap \elgen{v_1}{w_2},$ colour both edges with $c,$ note 
         that $c \notin \elgen{{v_2}}{u} \cup \elgen{v_2}{w_1}$
         by \eqref{equ:4-pyramid-1} and \eqref{equ:4-pyramid-2} (symmetry) and $c$ is on at most one of the lists of
         $\elgen{v_1}{w_1}$ and $\elgen{v_2}{w_2}$ by \eqref{equ:4-pyramid-2} (symmetry). In both cases we can we can 
         finish semi-greedily. So we have
   \begin{equation}\label{equ:4-pyramid-3}
   \elgen{u}{{w_1}} \cap \elgen{v_1}{w_2} = \emptyset
   \end{equation}
         If there is a colour $c \in \elgen{v_1}{w_2} \cap \elgen{v_1}{w_1},$ colour $\elgen{v_1}{w_2}$ with $c,$ note that 
         by \eqref{equ:4-pyramid-1}, \eqref{equ:4-pyramid-2} and \eqref{equ:4-pyramid-3} (symmetry) the only other edge that can have
         $c$ on its list is $\egen{v_1}{{u}}$ 
         and finish semi-greedily. Hence
      \begin{equation}\label{equ:4-pyramid-4}
   \elgen{v_1}{w_2} \cap \elgen{v_1}{w_1} = \emptyset.
   \end{equation}
         If there is a colour $c \in \elgen{v_1}{{u}} \cap \elgen{v_1}{w_2},$ colour $\elgen{v_1}{{u}}$ with $c,$ observe that by 
          \eqref{equ:4-pyramid-1}, \eqref{equ:4-pyramid-2}, \eqref{equ:4-pyramid-3} and \eqref{equ:4-pyramid-4} (symmetry) 
         the only edge that can have
         $c$ on its list is $\egen{{u}}{w_2}$ and finish semi-greedily. Consequently we have
   \begin{equation}\label{equ:4-pyramid-w_2}
   \elgen{{v_1}}{u} \cap \elgen{v_1}{w_2}= \emptyset.
   \end{equation}
  Colour the edges $\egen{v_1}{{u}},$ $\egen{{v_2}}{u},$ $\egen{{u}}{w_1}$ and $\egen{u}{{w_2}}$ semi-greedily. The remaining edges form a 4-cycle that retain at least
  two available colours by \eqref{equ:4-pyramid-w_2}. We finish by applying Theorem~\ref{thm:galvin}.
\end{proof}

 In the next two lemmas we will colour the instances related to the substructures shown in 
 Figure~\ref{fig:tw3md7-conf-2} and ~\ref{fig:tw3md7-conf-3}.

\begin{figure}
\centering
\tikzstyle{vertex}=[circle,draw, minimum size=14pt,inner sep=0pt]
\tikzstyle{edge} = [draw,-]
\tikzstyle{weight} = [font=\small,draw,fill           = white,
                                  text           = black]
\begin{subfigure}[b]{0.45\textwidth}\centering
 \begin{tikzpicture}[scale=1.1]
    \foreach \pos/\name in {{(1,1)/v_2}, {(-1,2)/v_1}, {(2,4)/v_3},
                            {(1,4)/{u}},{(-3,2)/w_1},{(2,1)/w_2},{(0,2)/w_3}}
       \node[vertex, align=center] (\name) at \pos {$\name$};
    \foreach \source/ \dest /\weight in {{u}/v_1/3, {u}/v_2/4, {u}/v_3/2, {u}/w_1/7, {u}/w_2/7, {u}/w_3/7,
                                          v_1/w_1/3, v_1/w_3/3, 
                                          v_2/w_1/4, v_2/w_2/4, v_2/w_3/4,
                                         v_3/w_2/2} 
       \path[edge] (\source) -- node[weight] {$\weight$} (\dest);
 \end{tikzpicture} \caption{}
\label{fig:cherry-configuration-one-max-deg-7}
\end{subfigure}
\begin{subfigure}[b]{0.45\textwidth}\centering
 \begin{tikzpicture}[scale=1.1]
    \foreach \pos/\name in {{(1,1)/{u}}, {(-1,-1)/v_1}, {(3,-1)/v_3},
                            {(1,3)/v_2},{(-0.2,1)/w_1},{(2.2,1)/w_2},{(1,-0.2)/w_3}}
       \node[vertex, align=center] (\name) at \pos {$\name$};
    \foreach \source/ \dest /\weight in {v_2/{u}/3, w_1/v_2/3, v_2/w_2/3, v_3/w_3/3,
                                          v_1/w_1/3, v_1/w_3/3, v_1/{u}/3, {u}/v_3/3,
                                          {u}/w_1/7, {u}/w_2/7, {u}/w_3/7,
                                         v_3/w_2/3} 
       \path[edge] (\source) -- node[weight] {$\weight$} (\dest);
 \end{tikzpicture} \caption{}
\label{fig:cherry-configuration-two-max-deg-7}
\end{subfigure}
\caption{The integers indicate the minimum sizes of the respective lists.}

\end{figure}
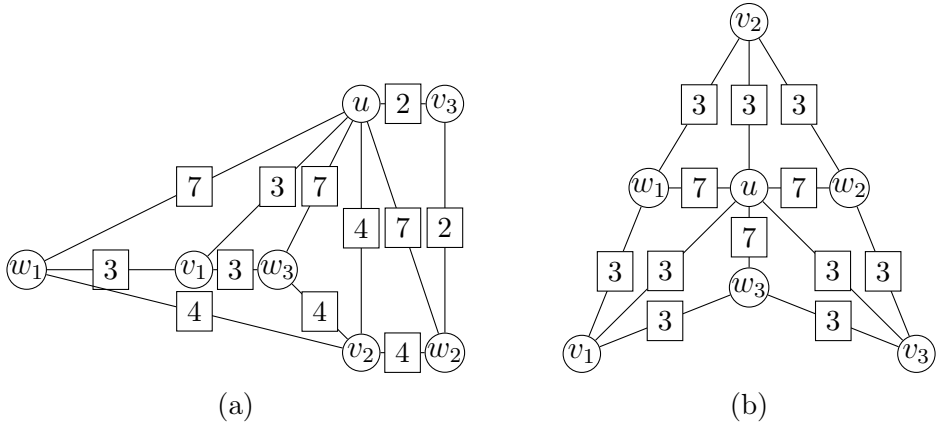

\begin{lemma}
\label{lem:cherry-configuration-one-max-deg-7}
 Let $G$ be the graph shown in Figure~\ref{fig:cherry-configuration-one-max-deg-7} with lists of colours assigned to 
 the edges, where the minimum sizes of the lists are indicated by the integers on the edges. Then there is an 
$L$-edge-colouring of $G$. 
\end{lemma}
\begin{proof}
 The same argument as in the proof of
  Lemma~\ref{lem:cariolaro} allows us to pick colours
  $c_1 \in \elgen{v_3}{u}$ and $c_2 \in \elgen{v_2}{w_2}$ such that either $c_1 = c_2$ or the lists 
  of ${}v_2u$ and $v_3w_2$ contain each at most one of the
  colours $c_1$ and $c_2.$ Colour $v_3u$ with $c_1$, $v_2w_2$ with $c_2$ and $v_3w_2$ semi-greedily. Colour the edge ${u}w_2$ with a colour 
  $c \in \elgen{{u}}{w_2} \setminus \elgen{{v_1}}{u}$ and
  apply Lemma~\ref{lem:4-pyramid} to the rest.
\end{proof}

\begin{lemma}
\label{lem:cherry-configuration-two-max-deg-7}
 Let $G$ be the graph shown in Figure~\ref{fig:cherry-configuration-two-max-deg-7} with lists of colours $L$ assigned to 
 the edges, where the minimum sizes of the lists are indicated by the integers on the edges. Then there is an 
 $L$-edge-colouring of $G$. 
\end{lemma}

\begin{proof} 
 If there is a colour $c \in \elgen{v_2}{{u}} \setminus \elgen{v_2}{w_1},$ colour
 $\egen{v_2}{{u}}$ with $c$ and colour the edges $\egen{v_1}{{u}}$ and $\egen{v_3}{{u}}$ semi-greedily. We can apply 
 Lemma~\ref{lem:cherry-configuration-bipartite} to colour the rest and get an $L$-edge-colouring. This and symmetry yield
 \begin{equation}
  \label{equ:cherry-configuration-two-max-deg-7-1}
  \elgen{v_2}{{u}} = \elgen{v_2}{w_1} = \elgen{v_2}{w_2}.
 \end{equation}
 If there is a colour $c \in \elgen{v_2}{{u}}$ that is not on $\elgen{{u}}{w_1}$ use it 
 to colour $\egen{v_2}{{u}}.$ If $c \in \elgen{{v_3}}{u},$ then we have $c \in \elgen{v_3}{w_2}$
 by \eqref{equ:cherry-configuration-two-max-deg-7-1} and symmetry. Colour $\egen{v_3}{w_2}$ with that colour. If
 $c \notin \elgen{{v_3}}{u}$ colour
 $\egen{v_3}{w_2}$ semi-greedily. In both cases the list of ${v_3}u$ retains at least two available colours and we continue 
 by colouring the edges $\egen{v_2}{w_2}$ and $\egen{v_2}{w_1}$ semi-greedily. Apply Lemma~\ref{lem:ballon} to
 colour the edges $\egen{v_1}{w_1},$ $\egen{v_1}{{u}},$ $\egen{v_1}{w_3},$ $\egen{{v_3}}{u}$ and $\egen{v_3}{w_3}$ and finish semi-greedily
 to get an $L$-edge-colouring. By this and symmetry we have
\begin{equation}
  \label{equ:cherry-configuration-two-max-deg-7-2}
  \elgen{v_2}{{u}} \subset (\elgen{{w_1}}{u} \cap \elgen{{u}}{w_2} \cap \elgen{{u}}{w_3}).
 \end{equation}
 By the size of the lists, \eqref{equ:cherry-configuration-two-max-deg-7-2} and symmetry, there is a colour 
 $c \in \elgen{v_2}{{u}}$ that is also in $\elgen{v_3}{{u}}$ or $\elgen{v_1}{{u}}.$ By symmetry we can assume the latter.
 By \eqref{equ:cherry-configuration-two-max-deg-7-1} we have also $c \in \elgen{v_2}{w_1}.$ Colour $\egen{v_1}{{u}}$ and $\egen{v_2}{w_1}$ with $c$ and colour
 the edges $\egen{v_1}{w_1},$ $\egen{v_1}{w_3}$ semi-greedily. Apply Lemma~\ref{lem:ballon} to colour the edges $\egen{v_3}{w_3},$ $\egen{v_3}{{u}},$ $\egen{{v_2}}{u},$
 $\egen{v_2}{w_2}$ and $\egen{v_3}{w_2}$ and finish semi-greedily to get an $L$-edge-colouring of $G$.

\end{proof}

\section{Proofs of the main results}
\label{sec:results}
 In this section we will combine the results of the two previous sections in order to give proofs of Theorem~\ref{thm:simple-results}
 and~\ref{thm:3-trees}.
 The size of a graph $G$ is $|V(G)|+|E(G)|$. $H$ is smaller than $G$ if its size is less than the size of $G$. 
 For a subset of vertices $W \subset V(G),$ we denote by $G\langle W \rangle$ the graph with vertex set $W \cup N(W)$ and edge set 
 $E(G) \setminus E(G-W)$. 
 
 Let $G$ be a graph with an assignment of lists 
 $L$ such that for a fixed $l \in \mathbb{N}$ each list $L(vw)$ has a size of at least $\max(l,\deg_G(v),\deg_G(w))$. Suppose 
 that for some proper 
 subset of vertices
 $W \subset V(G),$ we can find an $L$-edge-colouring $\mathcal{C}$ of the graph $G - W$. In order
 to extend $\mathcal{C}$ to an $L$-edge-colouring of $G$ we need find an $L^\mathcal{C}$-colouring of $G \langle W \rangle$.
 For an edge $w_1w_2 \in E(G)$ with $w_1,$ $w_2 \in W$ we have 
 \begin{equation}
 \label{equ:remaining-colours-1}
  |L^{\mathcal{C}}(w_1w_2)| = |L(w_1w_2)|.
 \end{equation}
 Since $\max(\deg_G(v),\deg_G(w),l) \geq \deg_G(v) = \deg_{G-W}(v) + \deg_{G\langle W \rangle}(v),$ we have for  
 an edge $vw \in E(G)$ with $w \in W$ and $v \in V(G) \setminus W$
 \begin{equation}
 \label{equ:remaining-colours-2}
  |L^{\mathcal{C}}(vw)|\geq |L(vw)| - \deg_{G-W}(v) \geq \deg_{G\langle W \rangle}(v)
 \end{equation} 
 Our proofs will be based on minimality. In most
 cases we want to prove for some subgraph-closed family of graphs and a fixed $l \in \mathbb{N}$
 that for each member $G$ with lists of colours $L(vw)$
  assigned to the edges
 of size at least $|L(vw)|\geq \max(l,\deg_G(v),\deg_G(w))$ 
 there is always an $L$-edge-colouring of $G$. 
 Suppose this does not hold for a graph $G$, but for all graphs that are smaller than $G$ and
 let $L$ be an assignment of lists of above described sizes for which there is no $L$-edge-colouring of $G$. Let $W \subset V(G)$ be some
 non-empty subset of vertices.
  As $\ch'(H) \leq \ch'(G)$ for every subgraph $H \subset G$, $G-W$ is smaller than $G$ and
 for each $v \in V(G-W)$ it holds that $\deg_{G-W}(v) \leq \deg_G(W),$ we can find
 an $L$-edge-colouring $\mathcal{C}$ of $G -W$. This yields
 immediately that $G$ has no isolated vertices and more interestingly for every edge $vw$ we have 
 \begin{equation}
  \label{equ:edge-degree}
  \deg(v)+\deg(w) \geq \max(l,\deg_G(v),\deg_G(w))+2.
 \end{equation} 
 Otherwise colour $G-vw$ by minimality and observe that $L(vw)$ retains at least one available colour by
 \eqref{equ:remaining-colours-2}, which contradicts the assumptions on $G$. 

\subsection{Proofs}
The next three lemmas imply Theorem~\ref{thm:simple-results}.
\begin{lemma}
 \label{lem:path-width-3-delta-6}
  Let $G$ be a graph of path-width at most 3 with an assignment of lists $L$ 
  to the edges of $G,$ such that each list $L(vw)$ has a size of at least
  $\max(6,\deg(v),\deg(w))$. Then $G$ has an $L$-edge-colouring.
\end{lemma}
 \begin{proof}
 We will assume that the lemma is wrong and obtain a contradiction. So let $G$ be a smallest counterexample to the lemma with lists $L(vw)$ of size at
 least $\max(6,\deg(v),\deg(w))$ assigned to the edges, such that there is no $L$-edge-colouring of $G$.
 By \eqref{equ:edge-degree} $G$, has minimum degree 2 and a $(3,6)$-substructure $(V,W,{u})$ as stated in Lemma~\ref{lem:conf-pw3md6}.
 If $G$ has the substructure of \ref{itm:conf-pw3md6-1}
 of Lemma~\ref{lem:conf-pw3md6} and
 hence $|W| \geq 3,$ choose a subset $W_1 \subset W$ of size 3 and an $L$-edge-colouring $\mathcal{C}_1$ of
 $G - W_1$ by minimality. Inequality \eqref{equ:remaining-colours-2} asserts that the lists of 
 remaining colours of the graph
 $G\langle W_1 \rangle$ retain sizes big enough to apply Lemma~\ref{lem:K(3,3)_Y-choosable}, to extend 
 $\mathcal{C}_1$ to an $L$-edge-colouring of $G$.
 Thus we can assume that $G$ has the substructure of Lemma~\ref{lem:conf-pw3md6}\ref{itm:conf-pw3md6-2} (see Figure~\ref{fig:pw3md6-conf-2}). 
 Set $W_2 = W \cup \{{u}\}$ and use as before 
 the minimality of $G$ to find an $L$-edge-colouring 
 $\mathcal{C}_2$ of the graph $G - W_2$. By \eqref{equ:remaining-colours-1} and \eqref{equ:remaining-colours-2}, the
 lists of remaining colours of the graph $G\langle W_2 \rangle$ retain sizes big enough to colour ${u}v_3$ semi-greedily
   and then
 apply Lemma~\ref{lem:4-pyramid} to extend $\mathcal{C}_2$ to an $L$-edge-colouring of $G$. A contradiction.
 \end{proof}
\begin{lemma}
 \label{lem:tree-width-3-delta-7}
  Let $G$ be a graph of tree-width at most 3 with an assignment of lists $L$ to the edges of $G,$ such that each list $L(vw)$ has a size of at least
  $\max(7,\deg(v),\deg(w))$. Then $G$ has an $L$-edge-colouring.
 \end{lemma} 
 \begin{proof}
  We will assume that the lemma is wrong and obtain a contradiction. So let $G$ be a smallest counterexample to the 
  lemma with lists $L(vw)$ of size at
  least $\max(7,\deg(v),\deg(w))$ assigned to the edges, such that there is no $L$-edge-colouring of $G$.
  We can assume that $G$ is connected and non-empty. 
  By \eqref{equ:edge-degree}, $G$ has minimum degree 2 and a $(3,7)$-substructure $(V,W,{u})$ as stated in Lemma~\ref{lem:conf-tw3md7}.
 If $G$ has the substructure of Lemma~\ref{lem:conf-tw3md7}\ref{itm:conf-tw3md7-1} and hence $|W| \geq 4,$ choose a subset $W_1 \subset W$ of 
 size 4. By minimality we can find an $L$-edge-colouring $\mathcal{C}_1$ of the graph
 $G - W_1$. Inequality \eqref{equ:remaining-colours-2} asserts that the size of the lists of remaining colours for the graphs
 $G\langle W_1 \rangle$ are big enough to apply Lemma~\ref{lem:K(4,4)_almost_Y-choosable}, to extend 
 $\mathcal{C}_1$ to an $L$-edge-colouring of $G$. 
 Thus we can assume that $G$ has one of the substructures of Lemma~\ref{lem:conf-tw3md7}\ref{itm:conf-tw3md7-1}.
 Set $W_2 = W \cup \{{u}\}$ and use as before the minimality of $G$ to find an $L$-edge-colouring 
 $\mathcal{C}_2$ of the graph $G - W_2$. By \eqref{equ:remaining-colours-1} and \eqref{equ:remaining-colours-2}, the
 lists of remaining colours of the graph $G\langle W_2 \rangle$ retain sizes big enough to
 apply Lemma~\ref{lem:K(3,3)_Y-choosable},~\ref{lem:cherry-configuration-one-max-deg-7} 
 or~\ref{lem:cherry-configuration-two-max-deg-7} respectively to extend $\mathcal{C}_2$ to an $L$-edge-colouring of $G$. A contradiction.
 \end{proof}

  We get Theorem~\ref{thm:3-trees} as a corollary.

\begin{proof}[Proof of Theorem~\ref{thm:3-trees}]
 Let $G$ be a 3-tree.
 If ${\Delta(G) \geq 7}$ 
 we have $\ch'(G) = \chi'(G)$ by Lemma~\ref{lem:tree-width-3-delta-7}.
 If $\Delta(G) \leq 6,$ then by Lemma~\ref{lem:small-3-trees} 
 $G$ has either path-width 3 or is isomorphic to one of the graphs shown in 
 Figure~\ref{fig:small-3-trees}. In the first case we can apply Lemma~\ref{lem:path-width-3-delta-6} and otherwise we are done by 
 Lemma~\ref{lem:small_graphs_of_tree-width_three}.
\end{proof}

\begin{lemma}
\label{lem:pw4md11}
  Let $G$ be a graph of tree-width at most 4 with an assignment of lists $L$ to the edges of $G,$ such that each list $L(vw)$ has a size of at least
  $\max(10,\deg(v),\deg(w))$. Then $G$ has an $L$-edge-colouring.
\end{lemma}

\begin{proof}
 We will assume that the lemma is wrong and obtain a contradiction. So let $G$ be a 
smallest counterexample to the lemma with lists $L(vw)$ of size at
  least $\max(10,\deg(v),\deg(w))$ assigned to the edges, such that there is no $L$-edge-colouring of $G$.
  We can assume that $G$ is connected and non-empty. 
  By \eqref{equ:edge-degree} $G$ has minimum degree 2 and a $(4,10)$-substructure $(V,W,{u})$ with a dedicated subset $W_1 \subset W$ of size 4
  as stated in Lemma~\ref{lem:conf-tw3md7}.
  If $G$ has the substructure of Lemma~\ref{lem:conf-pw4md10}\ref{itm:pw4-1}  
  and thus each element of $W_1$ has a degree of at least 3,
 pick an $L$-edge-colouring of the graph $G \setminus W_1$ by minimality of $G$. By \eqref{equ:remaining-colours-1} the
 size of the lists of
 remaining colours retain sizes big enough to find an $L$-edge-colouring of the graph $G \langle W_1 \rangle$ 
 by applying Lemma~\ref{lem:k44-min-deg-3}. If $G$ has the substructure of Lemma~\ref{lem:conf-pw4md10}\ref{itm:pw4-2} and hence each
 vertex of $W_1$ has a degree of at most
 3, we can find an $L$-edge-colouring of $G$ as seen in the proof of Lemma~\ref{lem:tree-width-3-delta-7} with 
 Lemma~\ref{lem:K(4,4)_almost_Y-choosable}. If $G$ has the substructure of Lemma~\ref{lem:conf-pw4md10}\ref{itm:pw4-3} and therefore 
  two vertices $w_1,$ $w_2 \in W_1$ of
 degree 2 have the same neighbourhood, we can find an $L$-edge-colouring of $G- w_1-w_2$ by minimality and extend this to
 an $L$-edge-colouring of $G$ by applying Theorem~\ref{thm:galvin} to the graph $G \langle \{w_1,w_2\} \rangle$ with lists
 of remaining colours.
 So we can assume that $G$ has the substructure of Lemma~\ref{lem:conf-pw4md10}\ref{itm:pw4-4} and by consequence
 $W_1$ contains exactly one vertex $w_1$ of degree 4 and three vertices $w_2,$ $w_3$ and $w_4$ 
 of degree 2 with pairwise distinct neighbourhoods. Let $V = \{v_1,v_2,v_3,u\}$ as shown
 in Figure~\ref{fig:pw4-4222}.
 Denote by $G_1$ the graph $G - W_1$ and let $G^*$ be the graph obtained from $G_1$ by adding two new vertices $p_1,$ $p_2$
 and connecting both to each vertex of $V $. Observe that $\Delta(G^*) \leq \Delta(G)$ and 
 $|V(G^*)|+|E(G^*)|<|V(G)|+|E(G)|$. Further $G^*$ has path-width 4. Lemma~\ref{lem:conf-pw4md10} provides a width 
 4 path decomposition
 of $G_1$ where $V = V_t$ for some vertex $t$ of the associated tree. We can extend this to a width 4 path decomposition
 of $G^*$ as in the proof of Lemma~\ref{lem:tree-width-k-configurations}. Let $L^*$ be an assignment of lists to the edges of $G^*$ with
 \begin{itemize}
  \item $L^*(e) = L(e)$ if $e \in E(G_1),$
  \item $L^*(up_j) = L(uw_1)$ and
  \item $L^*(v_ip_j) = L(v_iw_1)$
 \end{itemize}
  for $1 \leq j \leq 2$ and $1 \leq i \leq 3$. By minimality there is an $L^*$-colouring $\mathcal{C}^*$ of $G^*$.\footnote{At this point 
 it is important that $\ch'(G^*) \leq ch'(G)$ as it was discussed briefly in Section~\ref{sec:introduction}.}
 We can extract an $L$-edge-colouring $\mathcal{C}$ 
 of the graph $G_1$ by setting $\mathcal{C}(e) = \mathcal{C}^*(e)$ for each edge $e \in E(G_1)$. As the
 graph $G \langle W_1 \rangle$ with the lists of remaining colours $L^{\mathcal{C}}$ fulfils the conditions of Lemma 
 ~\ref{lem:pw4-ins}, we can extend $\mathcal{C}$ to an $L$-edge-colouring of $G$. A contradiction.
\end{proof}

  \subsection{Remarks}
\begin{figure}
\centering
\tikzstyle{vertex}=[circle,draw,minimum size=15pt,inner sep=0pt]
\tikzstyle{edge} = [draw,-]
\tikzstyle{weight} = [font=\small]
\begin{subfigure}[b]{0.3\textwidth}\centering
 \begin{tikzpicture}[scale=1]
    \foreach \pos/\name in {{(0,0)/v_1}, {(1,0)/v_2}, {(2,0)/v_3},{(3,0)/v_4},
                            {(0,2)/w_1}, {(1,2)/w_2}, {(2,2)/w_3},{(3,2)/w_4}}
       \node[vertex, align=center] (\name) at \pos {$\name$};
    \foreach \source/ \dest /\weight in {v_1/w_1/2, v_1/w_2/2,  v_1/w_3/2,
                                         v_2/w_1/2, v_2/w_3/2, v_2/w_2/2,
                                         v_3/w_1/2, v_3/w_4/2, v_3/w_3/2, v_3/w_2/2,
                                         v_4/w_2/3, v_4/w_3/3,v_4/w_4/3, v_4/w_1/3}
       \path[edge] (\source) -- node[weight] {} (\dest);
 \end{tikzpicture}
 \caption{}
 \label{fig:pw4-concl-1}
\end{subfigure}
\begin{subfigure}[b]{0.3\textwidth}\centering
 \begin{tikzpicture}[scale=1]
    \foreach \pos/\name in {{(0,0)/v_1}, {(1,0)/v_2}, {(2,0)/v_3},{(3,0)/v_4},
                            {(0,2)/w_1}, {(1,2)/w_2}, {(2,2)/w_3},{(3,2)/w_4}}
       \node[vertex, align=center] (\name) at \pos {$\name$};
    \foreach \source/ \dest /\weight in {v_1/w_1/2, v_1/w_2/2, w_1/v_4/4,
                                         v_2/w_1/2, v_2/w_3/2, 
                                         v_3/w_1/2, v_3/w_4/2,
                                         v_4/w_2/43, v_4/w_3/4,v_4/w_4/4}C
       \path[edge] (\source) -- node[weight] {} (\dest);
 \end{tikzpicture}
 \caption{}
 \label{fig:pw4-concl-2}
\end{subfigure}
\caption{Two bipartite graph that are not $\{v_1, v_2, v_3, v_4\}$-choosable.}
\label{fig:pw4-problem}
\end{figure}
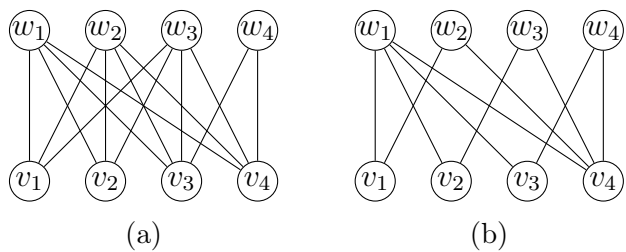

  The case of tree-width 3 and maximum degree 6 has been studied, but not resolved. Partial results can be found in~\cite{Lan12}.
  It would be nice to have
  more general versions of the lemmas concerning bipartite substructures in section~\ref{sec:intances-bip}. 
  These could be used in combination with Lemma~\ref{lem:tree-width-k-configurations}
  to colour graphs where the maximum degree is some linear function of the tree-width. Of course there may occur 
  substructures that do not permit edge-colouring from the lists of remaining colours. For
  example the graphs shown in Figure~\ref{fig:pw4-concl-1} and~\ref{fig:pw4-concl-2} are not $V$-choosable, while they do appear
  as substructures of graphs of path-width 4
  and maximum degree 10. But we can overcome these obstacles, by further analysis of the graph structure as seen in 
  Lemma~\ref{lem:conf-pw4md10} and using refined methods to colour the substructures as explained in 
  Lemma~\ref{lem:pw4-ins}.

\section*{Acknowledgement}
The author would like to thank Henning Bruhn for guidance and helpful remarks and Kristijan Radakovic for 
offering the gallery Nouvel Organon as work space.

\end{document}